\documentclass[12pt,leqno]{amsart}
\usepackage{amssymb,amsmath,amsthm,dsfont,fourier}

\usepackage{mathtools}
\usepackage{array}
\usepackage[text={6.6in,9in},centering]{geometry}
\usepackage[dvipsnames]{xcolor}
\usepackage{enumitem}
\usepackage{tikz-cd}
\usetikzlibrary{decorations.pathmorphing}
\usepackage[colorlinks=true,citecolor=NavyBlue,linkcolor=NavyBlue,urlcolor=NavyBlue]{hyperref}

\usepackage[all]{xy}
\usepackage{graphicx}
\usepackage{microtype}
\newdir{ >}{!/6pt/@{ }*:(1,0.2)@_{>}*:(1,-0.2)@^{>}}

\newcommand{\CRig}{\mathbf{CRig}}
\newcommand{\CRigr}{\mathbf{CRig}_{r}}
\newcommand{\SpecSp}{\mathbf{Spec}}
\newcommand{\CohFrm}{\mathbf{CohFrm}}
\newcommand{\Frm}{\mathbf{Frm}}
\newcommand{\Loc}{\mathbf{Loc}}
\newcommand{\AlgLat}{\mathbf{AlgLat}}
\newcommand{\DLat}{\mathbf{DLat}_{0,1}}
\newcommand{\DLatFrm}{\mathbf{DLat}_{\mathrm{frm}}}
\newcommand{\QuantTop}{\mathbf{CQuant}_{\mathrm{int}}}
\newcommand{\CIdRig}{\mathbf{CIdRig}}
\newcommand{\Supp}{\mathbf{Supp}}
\newcommand{\FSupp}{\mathbf{FSupp}}
\newcommand{\OSupp}{\mathbf{OSupp}}
\newcommand{\TotSupp}{\mathbf{Supp}_{\mathrm{tot}}}
\newcommand{\Sh}{\mathsf{Sh}}

\newcommand{\Spec}{\mathsf{Spec}}
\newcommand{\Speck}{\mathsf{Spec}_{k}}
\newcommand{\Specl}{\mathsf{Spec}_{\mathfrak l}}
\newcommand{\RId}{\mathsf{RId}}
\newcommand{\RIdk}{\mathsf{RId}_{k}}
\newcommand{\Id}{\mathsf{Id}}
\newcommand{\Idk}{\mathsf{Id}_{k}}
\newcommand{\KOpen}{\mathfrak K_{0}}
\newcommand{\Coz}{\operatorname{Coz}}
\newcommand{\Zar}{\operatorname{Zar}}
\newcommand{\Max}{\operatorname{Max}}
\newcommand{\Maxk}{\operatorname{Max}_{k}}
\newcommand{\op}{\mathrm{op}}

\newtheorem{theorem}{Theorem}[section]
\newtheorem{proposition}[theorem]{Proposition}
\newtheorem{lemma}[theorem]{Lemma}
\newtheorem{corollary}[theorem]{Corollary}
\theoremstyle{definition}
\newtheorem{definition}{Definition}[section]
\newtheorem{example}[theorem]{Example}
\numberwithin{equation}{section}
\theoremstyle{remark}
\newtheorem{remark}[theorem]{Remark}

\def\leq{\leqslant}

\hypersetup{
  pdftitle={Radical-Ideal Functors, a Support Bifibration, and Quantale Completion for Commutative Semirings},
  pdfauthor={Pronay Biswas, Amartya Goswami, and Sujit Kumar Sardar},
  pdfsubject={Categorical radical-ideal theory, support bifibrations, spectral comparison, and monadic quantale completion for commutative semirings},
  pdfkeywords={commutative semiring, radical ideal, coherent frame, spectral space, Grothendieck bifibration, quantale monad}
}

\begin{document}
\raggedbottom

\title[Radical-ideal functors and quantale completion]{Radical-Ideal Functors, a Support Bifibration, and Quantale Completion for Commutative Semirings}

\author{Pronay Biswas}
\thanks{Corresponding author: Pronay Biswas, \href{mailto:pronayb.math.rs@jadavpuruniversity.in}{pronayb.math.rs@jadavpuruniversity.in}.}
\address{Department of Mathematics, Jadavpur University, 188 Raja Subodh Chandra Mallick Road, Jadavpur, Kolkata, West Bengal 700032, India}
\email{pronayb.math.rs@jadavpuruniversity.in, pronaybiswas1729@gmail.com}
\urladdr{ORCID: \url{https://orcid.org/0009-0008-2710-1748}}

\author{Amartya Goswami}
\address{Department of Mathematics and Applied Mathematics, University of Johannesburg, P.O. Box 524, Auckland Park 2006, South Africa; and National Institute for Theoretical and Computational Sciences (NITheCS), South Africa}
\email{agoswami@uj.ac.za}
\urladdr{ORCID: \url{https://orcid.org/0000-0003-4829-0847}}

\author{Sujit Kumar Sardar}
\address{Department of Mathematics, Jadavpur University, 188 Raja Subodh Chandra Mallick Road, Jadavpur, Kolkata, West Bengal 700032, India}
\email{sujitk.sardar@jadavpuruniversity.in, sksardarjumath@gmail.com}
\urladdr{ORCID: \url{https://orcid.org/0000-0001-7837-0835}}


\begin{abstract}
We organize ordinary, subtractive ($k$-), and strong ideal theory of commutative semirings into a functorial framework. Radical extension is left adjoint to contraction and yields coherent-frame-valued functors naturally represented by the open-set frames of the corresponding prime spectra. The comparison from ordinary to $k$-radical ideals is a natural nucleus whose components are surjective and, under coherent Stone duality, correspond to dense sublocale embeddings. Ordinary, $k$-, and strong prime spectra form nested natural spectral functors, while universal support objects recover the spectra, radical frames, and complemented idempotents. Finite supports assemble into a Grothendieck bifibration with a canonical bicartesian section. For complete idealic semirings, $k$-ideal completion realizes a subtractive form of ideal quantale completion. We compute the induced monad, identify its restriction to frames with the classical ideal-lattice monad, and prove that its Eilenberg--Moore category is equivalent to the category of integral commutative quantales. Applications include a Stone-spectrum criterion for positive cones of $f$-rings and density criteria for $k$-prime spectra of $r$-semirings.
\end{abstract}

\maketitle

\tableofcontents

\section{Introduction}

A single commutative semiring supports several inequivalent ideal doctrines. Ordinary ideals, subtractive ideals (usually called $k$-ideals), and strong ideals have different radical reflections and different prime spectra. The categorical question is therefore not only how to construct these objects, but how they vary with homomorphisms, how the closure doctrines compare naturally, and which universal constructions represent them. Our point of departure is the radical-ideal functor of Banaschewski and its frame-theoretic developments \cite{Banaschewski,Dube18,Martinez}, together with coherent Stone duality \cite{Hochster,Johnstone,Picado,SpectralBook} and the ideal and spectral theory of commutative semirings \cite{Golan,Lescot,AG24,Ray22}.

For a homomorphism of semirings, the set-theoretic image of an ideal need not be an ideal of the required kind. Radical extension, rather than direct image, supplies the covariant operation; it is characterized as the left adjoint of contraction. This observation produces coherent-frame-valued radical-ideal functors and places ordinary and $k$-prime spectra on the dual side of coherent Stone duality. The passage from ordinary radical ideals to $k$-radical ideals then becomes a natural nucleus, not merely an objectwise closure operation.

Support theory supplies a complementary universal language. Reticulation and support constructions \cite{Simmons,JoyalSupport,Johnstone} represent algebraic information in bounded distributive lattices, frames, and spectral spaces. Related categorical support formalisms occur in tensor triangular geometry, point-free reconstruction, and universal algebra \cite{Balmer07,Kock,Georgescu21,Krause24}. We use support objects both to reconstruct the spectra and radical frames and to organize change of scalars: finite supports form the fibers of a Grothendieck bifibration over the semiring category, with the canonical compact-open support giving a bicartesian section.

Ideal completion as a free quantale construction for idempotent or ordered semirings is established in \cite{NishizawaFurusawa14,Fujii23}. We study its subtractive incarnation for complete idealic semirings: the free object is formed from $k$-ideals with their multiplicative $k$-ideal product. The point is not to claim the underlying ideal-completion adjunction de novo, but to compute the induced monad in this setting, identify its Eilenberg--Moore algebras, and relate it to the radical and spectral functors developed in the earlier sections. On the full subcategory of frame objects, the resulting monad agrees canonically with the classical ideal-lattice monad on bounded distributive lattices; that monad and its iterates are studied in \cite{Johnstone,Razafindrakoto25}. Outside the frame subcategory, the two constructions have different bases and retain different multiplicative information.

Several objectwise ingredients used below are classical or adapt standard arguments from semiring ideal theory, frame theory, and coherent duality. The main contribution of the paper is concentrated in the categorical structures that connect these ingredients. 

First, we construct the coherent-frame-valued $k$-radical functor and compare it with the ordinary radical-ideal functor by a natural nucleus; under coherent Stone duality this comparison is identified with restriction of opens along the natural inclusion of the $k$-prime spectrum. In the same framework, the ordinary, $k$-, and strong prime spectra are assembled into nested natural spectral functors, together with closed embeddings of their constructible refinements. 

Second, we globalize finite support theory: the support categories form the fibers of a Grothendieck bifibration over the semiring category, and the canonical compact-open support defines a bicartesian section and a quasi-inverse to the projection. 

Third, we compute the monad induced by subtractive $k$-ideal completion, identify its exact restriction to frame objects with the classical ideal-lattice monad, and characterize its Eilenberg--Moore algebras as integral commutative quantales. To the best of our knowledge, this combined natural comparison, support bifibration, and monadic identification does not occur in the published literature. The localic and sheaf-theoretic statements are recorded as formal categorical consequences of the comparison maps, while the Stone-space and density criteria are applications of the structural theory.

Our main results are as follows.
\begin{enumerate}[label=\upshape(\roman*)]
\item Extending the classical ordinary radical-ideal construction, radical extension is left adjoint to contraction for ordinary and $k$-radical ideals. These adjunctions define coherent-frame-valued functors \(\RId\) and \(\RIdk\), naturally represented by
\[
\RId(S)\cong\Omega(\Spec(S)),
\qquad
\RIdk(S)\cong\Omega(\Speck(S)).
\]

\item The $k$-radical closure is a natural nucleus
\[
\mathfrak j\colon\RId\Longrightarrow\RIdk.
\]
Under the standing reduced and conical hypotheses, its components are surjective and correspond under coherent Stone duality to dense sublocale embeddings. Ordinary, $k$-, and strong prime spectra form nested natural spectral functors, and their constructible refinements are closed embeddings.

\item Universal support objects recover the prime spectrum, the radical-ideal frame, and the Boolean algebra of complemented idempotents. Finite support categories assemble into a Grothendieck bifibration
\[
\pi\colon\TotSupp\longrightarrow\CRig.
\]
The canonical compact-open support is a bicartesian section and a quasi-inverse to \(\pi\).

\item In the complete idealic setting, the $k$-ideal construction realizes the subtractive form of ideal quantale completion. We compute the induced monad, identify its exact restriction to frame objects with the classical ideal-lattice monad, and prove that its Eilenberg--Moore algebras are precisely the complete idealic semirings whose multiplication distributes over arbitrary joins.

\item The structural results yield a Stone-spectrum criterion for positive cones of $f$-rings and density criteria for the natural inclusion of the $k$-prime spectrum of an $r$-semiring.
\end{enumerate}

Figure~\ref{fig:categorical-framework} summarizes the principal categorical constructions and their relations. The upper adjunction is Theorem~\ref{thm:k-ideal-quantale-adjunction}, its comparison with the ideal-lattice monad is Proposition~\ref{prop:monad-frame-comparison}, its Eilenberg--Moore comparison is Theorem~\ref{thm:EM-quantales}, and the two duality equivalences are recalled in Remark~\ref{rem:stone-ideal-completion}. The functors in the top diagram are established in Proposition~\ref{prop:spectrum-functors}, Theorem~\ref{thm:radical-functor}, Proposition~\ref{prop:k-radical-functor}, Definition~\ref{def:canonical-support-functor}, and Proposition~\ref{prop:strong-spectrum-functor}. The nested spectral transformations and their constructible refinements are Propositions~\ref{prop:nested-spectra} and~\ref{prop:constructible-spectra}. The middle square combines Propositions~\ref{prop:natural-spectral-representations} and~\ref{prop:natural-k-radical-comparison}; its vertical arrows are natural isomorphisms, and each component \(\mathfrak j_S\) is surjective and dual to a dense sublocale embedding. The reconstruction isomorphisms follow from Proposition~\ref{prop:natural-spectral-representations} and Remark~\ref{rem:natural-support-reconstruction}. The universal arrows are Proposition~\ref{prop:universal-supports}, while the global bifibration, its canonical bicartesian section, and the resulting equivalence are Theorem~\ref{thm:support-bifibration}. Finally, \(\CRigr\) is the full subcategory introduced after Definition~\ref{def:r-semiring}, whereas Proposition~\ref{prop:r-semiring-density} is deliberately stated in the larger category of commutative unital semirings.

The paper is organized as follows. Section~\ref{sec:preliminaries} fixes the categorical, ideal-theoretic, and spectral conventions. Section~\ref{sec:radical-functors} constructs the radical-ideal functors and their natural comparison. Section~\ref{sec:support-functors} develops support representations, compares the three spectra, proves the $r$-semiring density and Stone-space criteria, and constructs the support bifibration. Section~\ref{sec:quantale-completion} treats complete idealic semirings, the $k$-ideal quantale adjunction, the induced monad, and its Eilenberg--Moore category.

\medskip
\noindent\textbf{Acknowledgements.}
The first author is grateful to Themba Dube and Partha Pratim Ghosh for fruitful discussions and thanks the University Grants Commission (India) for a Senior Research Fellowship (ID: 211610013222/Joint CSIR--UGC NET June 2021).

\clearpage
\begin{figure}[!t]
\begin{minipage}{\textwidth}
The following figure separates the one-categorical backbone of the paper from its natural comparison and universal support statements. Here \(\mathbb T=\mathcal U\Idk\) is the monad induced by the quantale adjunction. The symbol \(\simeq\) denotes an equivalence of categories, \(\dashv\) denotes an adjunction, and hooked arrows denote the displayed faithful embeddings; no fullness is implied unless explicitly stated. The faithful functor
\[
\mathcal Q_{\wedge}\colon\CohFrm\longrightarrow\QuantTop,
\qquad
F\longmapsto(F,\vee,\wedge,0,1),
\]
regards a coherent frame as an integral commutative quantale with multiplication given by meet. The three spectrum functors are displayed covariantly by taking values in \(\SpecSp^{\op}\).
\end{minipage}
\medskip
\centering
\resizebox{\textwidth}{!}{%
\begin{tikzcd}[ampersand replacement=\&, column sep=4.2em, row sep=5.0em]
\& \& \CIdRig
  \arrow[rr,bend left=18,"\Idk"]
  \arrow[rr,phantom,"\dashv" description]
\& \& \QuantTop
  \arrow[ll,bend left=18,"\mathcal U"]
  \arrow[r,"\simeq"]
\& \CIdRig^{\mathbb T} \\
\CRigr \arrow[r,hook] \&
\CRig
  \arrow[r,shift left=2.2ex,"\Spec"]
  \arrow[r,"\Speck" description]
  \arrow[r,shift right=2.2ex,"\Specl"']
  \arrow[rr,bend left=42,"\RId"]
  \arrow[rr,bend right=42,"\RIdk"']
  \arrow[rrr,bend right=58,"\mathfrak D"']
\&
\SpecSp^{\op}
  \arrow[r,shift left=1.4ex,"\Omega"]
  \arrow[r,phantom,"\simeq" description]
\&
\CohFrm
  \arrow[l,shift left=1.4ex,"\Sigma"]
  \arrow[r,shift left=1.4ex,"\mathfrak k(-)"]
  \arrow[r,phantom,"\simeq" description]
  \arrow[rr,hook,bend left=40]
  \arrow[ur,hook,"\mathcal Q_{\wedge}" description]
\&
\DLat
  \arrow[l,shift left=1.4ex,"\mathsf{Idl}"]
\&
\AlgLat
\end{tikzcd}}

\medskip
\begin{minipage}[t]{0.59\textwidth}
\centering
\begin{tikzcd}[column sep=5.1em, row sep=3.0em]
\RId
  \arrow[r,Rightarrow,"\mathfrak j"]
  \arrow[d,Rightarrow,"\delta"']
&
\RIdk
  \arrow[d,Rightarrow,"\delta^{k}"]
\\
\Omega\circ\Spec
  \arrow[r,Rightarrow,"\Omega(\iota^{\op})"']
&
\Omega\circ\Speck
\end{tikzcd}
\end{minipage}\hfill
\begin{minipage}[t]{0.37\textwidth}
\centering
\small
\[
\begin{gathered}
\Specl\xRightarrow{\lambda}\Speck\xRightarrow{\iota}\Spec,\\[-1pt]
\mathfrak D\xRightarrow{\ \sim\ }\mathfrak k(-)\circ\RId,\\[-1pt]
\mathsf{Idl}\circ\mathfrak D\xRightarrow{\ \sim\ }\RId,\\[-1pt]
\mathfrak D=\KOpen\circ\Spec.
\end{gathered}
\]
\end{minipage}

\smallskip
\resizebox{0.98\textwidth}{!}{$
\begin{array}{c@{\qquad\qquad}c@{\qquad\qquad}c}
(\RId^{\mathrm{fin}}(S),[-])\xrightarrow{\ \exists!\ }(L,\mathfrak E)
&
(\RId(S),[-])\xrightarrow{\ \exists!\ }(F,\mathfrak E)
&
(X,d)\xrightarrow{\ \exists!\ }(\Spec(S),D)
\\[-1pt]
(L,\mathfrak E)\in\Supp(S)
&
(F,\mathfrak E)\in\FSupp(S)
&
(X,d)\in\OSupp(S)
\\[-1pt]
\Supp(S)\simeq\mathbf 1
&
&
\end{array}
$}

\smallskip
\[
\begin{gathered}
\pi\colon\TotSupp\longrightarrow\CRig
\quad\text{is a Grothendieck bifibration},\\[-1pt]
\mathsf s\colon\CRig\longrightarrow\TotSupp
\quad\text{is a bicartesian section and a quasi-inverse to }\pi.
\end{gathered}
\]
\caption{The categorical framework. Top: the principal functors, equivalences, adjunction, monadic comparison, and faithful embeddings. Middle: the nested spectral functors, natural spectral representations, $k$-radical comparison, and support reconstructions. Bottom: the universal support arrows for a fixed semiring and the global support bifibration.}
\label{fig:categorical-framework}
\end{figure}
\clearpage

\section{Categorical and spectral preliminaries}\label{sec:preliminaries}

This section fixes the three kinds of data used throughout the paper. We first isolate the ordinary, $k$-, and strong ideal closure conditions; we then organize prime spectra contravariantly and prove the finite-cover criteria needed for compactness; finally, we recall the frame, locale, and quantale categories in which the later constructions take values.

Throughout Sections~\ref{sec:preliminaries}--\ref{sec:support-functors}, semirings are assumed to be commutative, unital, conical, and reduced, and all homomorphisms preserve both $0$ and $1$. We write \(\CRig\) for the resulting category. A semiring is \emph{conical} if \(x+y=0\) implies \(x=y=0\), and \emph{reduced} if its only nilpotent element is $0$. Positive cones of partially ordered rings provide standard examples of conical semirings; see \cite{Fuchs}. The hypotheses used in Section~\ref{sec:quantale-completion} are stated separately. We use ordinary right-to-left composition, and we use adjunctions between posets in their Galois-connection form. General categorical terminology follows \cite{MacLane}.

\subsection{Ideal closure operators}

We begin by separating the closure doctrines that replace subtraction in semiring ideal theory. The purpose of this subsection is to fix the ordinary and $k$-ideal products, the corresponding radical reflections, and the least-object properties that will later make extension maps into left adjoints.

For a semiring $S$, let \(\Id(S)\) denote its lattice of ideals, ordered by inclusion, and let \(U(S)\) be its group of units. We write \(\Max(S)\) for the maximal ideals of $S$.

An ideal $I$ is a \emph{$k$-ideal} if
\[
a+b\in I\ \text{ and }\ a\in I\quad\Longrightarrow\quad b\in I
\]
for all $a,b\in S$; this is the subtractivity condition of Henriksen \cite{Hen58}. An ideal is \emph{strong}, or an \(\mathfrak l\)-ideal, if \(a+b\in I\) implies \(a,b\in I\). We denote the poset of $k$-ideals by \(\Idk(S)\), and the maximal $k$-ideals by \(\Maxk(S)\).

The inclusion \(\Idk(S)\hookrightarrow\Id(S)\) is reflective. Its reflector is the $k$-closure operator
\[
\mathcal C_k(I)
 =\{a\in S\mid a+b\in I\text{ for some }b\in I\}.
\]
Equivalently, \(\mathcal C_k(I)\) is the least $k$-ideal containing $I$. Intersections of $k$-ideals are therefore $k$-ideals, whereas finite sums need not be: in \(\mathds{N}\), the sum \(2\mathds{N}+3\mathds{N}=\mathds{N}\setminus\{1\}\) is not a $k$-ideal. Strong ideals are closely related to order ideals in lattice-ordered algebra; see \cite{Smith}.

For ideals $I,J\subseteq S$, their \emph{ordinary ideal product} is
\[
I\cdot J=\langle xy\mid x\in I,\ y\in J\rangle.
\]
If $I$ and $J$ are $k$-ideals, their \emph{$k$-ideal product} is
\[
I\odot_k J=\mathcal C_k(I\cdot J).
\]
Thus \(I\odot_kJ\) is the least $k$-ideal containing all products $xy$ with $x\in I$ and $y\in J$; this is the product convention of \cite[Section~2]{AG24}. An ideal $P$ is \emph{semiprime} if \(A\cdot A\subseteq P\) implies \(A\subseteq P\) for every ideal $A$.

The ordinary radical of an ideal is
\[
\sqrt I=\{x\in S\mid x^n\in I\text{ for some }n\geq1\}.
\]
An ideal is \emph{radical} if \(I=\sqrt I\). We write
\[
[a]=\sqrt{\langle a\rangle},
\qquad
[a_1,\ldots,a_n]=\sqrt{\langle a_1,\ldots,a_n\rangle}.
\]
The radical operator is a closure operator on \(\Id(S)\), and its fixed points form the poset \(\RId(S)\).

A proper ideal is \emph{prime} if \(ab\in P\) implies \(a\in P\) or \(b\in P\). A \emph{$k$-prime ideal} is a prime ideal that is also a $k$-ideal. We write
\[
\Spec(S)=\{P\mid P\text{ is prime}\},
\qquad
\Speck(S)=\{P\mid P\text{ is $k$-prime}\}.
\]
For a subset $A\subseteq S$, set
\[
\mathcal V_S(A)=\{P\in\Spec(S)\mid A\subseteq P\},
\qquad
\mathcal V_{k,S}(A)=\{P\in\Speck(S)\mid A\subseteq P\}.
\]
We suppress the subscript $S$ when no ambiguity can arise. The ordinary radical has the following prime-separation description; we include the proof because it is used throughout the spectral and support constructions.

\begin{lemma}[Prime separation]\label{lem:ordinary-radical-prime-intersection}
For every ideal $I$ of a commutative unital semiring $S$,
\[
\sqrt I=\bigcap_{P\in\mathcal V(I)}P,
\]
where an empty intersection is interpreted as $S$.
\end{lemma}

\begin{proof}
Every prime ideal is radical: if $x^n\in P$ for some $n\geq1$, repeated use of primality gives $x\in P$. Hence \(\sqrt I\) is contained in every prime ideal containing $I$.

Conversely, let $x\notin\sqrt I$ and put
\[
M_x=\{1,x,x^2,\ldots\}.
\]
Then $I\cap M_x=\varnothing$. Let \(\Sigma\) be the set of ideals $J$ such that
\(I\subseteq J\) and \(J\cap M_x=\varnothing\), ordered by inclusion. The union of a chain in \(\Sigma\) is again an ideal in \(\Sigma\), so Zorn's lemma gives a maximal member $P$. It is proper and does not contain $x$.

We claim that $P$ is prime. Suppose that $ab\in P$ while $a,b\notin P$. The ideals
\(P+\langle a\rangle\) and \(P+\langle b\rangle\) properly contain $P$, so maximality yields integers $m,n\geq0$, elements $p,q\in P$, and coefficients $r,s\in S$ with
\[
x^m=p+ra,
\qquad
x^n=q+sb.
\]
Multiplying gives
\[
x^{m+n}=pq+psb+qra+rsab\in P,
\]
contrary to \(P\cap M_x=\varnothing\). Thus $P$ is a prime ideal containing $I$ but not $x$. Therefore every element outside \(\sqrt I\) is excluded by some prime ideal containing $I$, proving the reverse inclusion.
\end{proof}

In parallel with Lemma~\ref{lem:ordinary-radical-prime-intersection}, define
\[
\sqrt[k]{I}=\mathcal R_k(I)=\bigcap_{P\in\mathcal V_k(I)}P,
\]
again interpreting an empty intersection as $S$. The fixed points of \(\sqrt[k]{(-)}\) are the \emph{$k$-radical ideals}; their poset is denoted by \(\RIdk(S)\). This definition is made for every ideal; one need not assume in advance that the ideal is a $k$-ideal.

\begin{lemma}\label{lem:k-radical-k-closure}
For every ideal $I$ of $S$,
\[
\sqrt[k]{I}=\sqrt[k]{\mathcal C_k(I)}.
\]
Thus the $k$-radical reflection factors through the $k$-ideal reflection.
\end{lemma}

\begin{proof}
The ideals $I$ and \(\mathcal C_k(I)\) are contained in exactly the same $k$-prime ideals. One implication is immediate. Conversely, if a $k$-prime ideal $P$ contains $I$, then monotonicity of $k$-closure gives
\(\mathcal C_k(I)\subseteq\mathcal C_k(P)=P\); compare \cite[Lemma~2.1(5)]{AG24}. Intersecting the common family of containing $k$-prime ideals gives the equality.
\end{proof}

A $k$-ideal $P$ is \emph{$k$-semiprime} if
\[
I\odot_k I\subseteq P\quad\Longrightarrow\quad I\subseteq P
\]
for every $k$-ideal $I$. We shall use the facts that a $k$-ideal is $k$-semiprime precisely when it is semiprime, and precisely when it is $k$-radical; moreover, the $k$-radical of a $k$-ideal is the least $k$-semiprime ideal containing it \cite[Proposition~3.8, Theorem~3.14, Corollary~3.15]{AG24}. For an arbitrary ideal $I$, Lemma~\ref{lem:k-radical-k-closure} gives
\[
\sqrt[k]{I}=\sqrt[k]{\mathcal C_k(I)},
\]
and every $k$-semiprime ideal containing $I$ also contains \(\mathcal C_k(I)\); hence the same least-object statement applies to arbitrary ideals after $k$-closure. In particular, conicality makes $(0)$ a $k$-ideal, while reducedness makes it semiprime; hence $(0)$ is a $k$-radical ideal.

\subsection{Prime spectra as contravariant functors}

We next topologize the ordinary and $k$-prime spectra and record their behavior under inverse image. Besides establishing contravariance, we prove finite-cover criteria for basic opens; these criteria are the compactness input for the coherent-frame and support constructions below.

A topological space is \emph{spectral} if it is quasi-compact and $T_0$, has a basis of quasi-compact open sets closed under finite intersections, and is sober \cite{Hochster,SpectralBook}. Let \(\KOpen(X)\) denote the bounded distributive lattice of quasi-compact open subsets of a spectral space $X$. We write \(\SpecSp\) for the category of spectral spaces and spectral maps.

The Zariski topology on \(\Spec(S)\) has basic opens
\[
D_S(a)=\{P\in\Spec(S)\mid a\notin P\}
\]
and basic closed sets \(\mathcal V_S(a)=\{P\mid a\in P\}\). The subspace \(\Speck(S)\) has basic opens
\[
D_{k,S}(a)=D_S(a)\cap\Speck(S).
\]
Both spectra are spectral spaces, and the displayed basic opens form bases of quasi-compact open sets closed under finite intersections; see \cite[Proposition~3.6]{Ray22}. We suppress the subscript $S$ when no ambiguity is possible.

\begin{proposition}\label{prop:spectrum-functors}
Inverse image of ideals defines contravariant functors
\[
\Spec,\Speck\colon\CRig^{\op}\longrightarrow\SpecSp.
\]
Equivalently, they may be regarded as covariant functors \(\CRig\to\SpecSp^{\op}\). For a morphism \(f\colon S\to T\), both functors act by
\[
f^*\colon P\longmapsto f^{-1}(P).
\]
\end{proposition}

\begin{proof}
The inverse image of a prime ideal under a unital semiring homomorphism is prime. If $P$ is a $k$-ideal and \(f(a+b)\in P\) with \(f(a)\in P\), then \(f(b)\in P\), so \(f^{-1}(P)\) is a $k$-ideal. Moreover,
\[
(f^*)^{-1}(D_S(a))=D_T(f(a)),
\qquad
(f^*)^{-1}(D_{k,S}(a))=D_{k,T}(f(a)).
\]
Thus the induced maps are spectral. The identity and composition laws are those of inverse image.
\end{proof}

The inclusions \(\iota_S\colon\Speck(S)\hookrightarrow\Spec(S)\) are the components of a natural transformation \(\iota\colon\Speck\Longrightarrow\Spec\) between the contravariant functors \(\CRig^{\op}\to\SpecSp\). If the same assignments are viewed covariantly as functors \(\CRig\to\SpecSp^{\op}\), the corresponding transformation has the reversed direction. Applying the open-set functor then produces the covariant comparison of radical-ideal frames constructed in Section~\ref{sec:radical-functors}.

\begin{theorem}\label{thm:basic-open-cover}
For \(a\in S\) and \(A\subseteq S\), the following conditions are equivalent:
\begin{enumerate}
\item \(D(a)\subseteq\bigcup_{b\in A}D(b)\);
\item \(\mathcal V(A)=\bigcap_{b\in A}\mathcal V(b)\subseteq\mathcal V(a)\);
\item \(a\in\sqrt{\langle A\rangle}\);
\item there exists a finite subset \(A_0\subseteq A\) such that \(a\in\sqrt{\langle A_0\rangle}\).
\end{enumerate}
\end{theorem}

\begin{proof}
The first two conditions are equivalent by taking complements. The equivalence of (2) and (3) is the prime-ideal characterization of the radical. Finally, every element of \(\langle A\rangle\) belongs to the ideal generated by a finite subset of $A$, which proves the equivalence of (3) and (4).
\end{proof}

\begin{corollary}\label{cor:k-basic-open-cover}
For \(a\in S\) and \(A\subseteq S\), the following conditions are equivalent:
\begin{enumerate}
\item \(D_k(a)\subseteq\bigcup_{b\in A}D_k(b)\);
\item \(\mathcal V_k(A)\subseteq\mathcal V_k(a)\);
\item \(a\in\sqrt[k]{\langle A\rangle}\);
\item there exists a finite subset \(A_0\subseteq A\) such that
\(a\in\sqrt[k]{\langle A_0\rangle}\).
\end{enumerate}
\end{corollary}

\begin{proof}
The first two conditions are equivalent by taking complements in \(\Speck(S)\), and the equivalence of (2) and (3) is immediate from the definition of the $k$-radical as an intersection of containing $k$-prime ideals. If these conditions hold, the quasi-compact basic open \(D_k(a)\) is covered by the family \(\{D_k(b)\mid b\in A\}\); hence a finite subfamily already covers it. Applying the first three equivalences to that finite subfamily yields (4). The implication (4)$\Rightarrow$(3) follows from monotonicity of $k$-radicalization.
\end{proof}

\subsection{Frames, coherent duality, and quantales}

This subsection records the target categories and dualities used in the rest of the paper. We recall coherent frames and their compact parts, the locale interpretation of nuclei, and the relation between spectral spaces, bounded distributive lattices, and quantales.

A \emph{quantale} is a complete lattice \((Q,\vee)\) equipped with an associative multiplication $*$ that distributes over arbitrary joins in each variable. It is \emph{commutative} when $*$ is commutative and \emph{unital} when it has a unit. Quantale homomorphisms preserve arbitrary joins and multiplication, and in the unital case also the unit; see \cite{Ros90}.

A \emph{frame} is a complete lattice in which finite meets distribute over arbitrary joins. Frame homomorphisms preserve arbitrary joins and finite meets. Every frame homomorphism \(h\colon L\to M\) has a right adjoint
\[
h_*(y)=\bigvee\{x\in L\mid h(x)\leq y\}.
\]
We write \(\Frm\) for the category of frames and frame homomorphisms and \(\Loc=\Frm^{\op}\) for the category of locales; see \cite{Johnstone,Picado}.

An element $c$ of a complete lattice is \emph{compact} if
\(c\leq\bigvee X\) implies \(c\leq\bigvee X_0\) for some finite \(X_0\subseteq X\). A frame is \emph{algebraic} if every element is a join of compact elements below it. It is \emph{coherent} if it is algebraic, its top is compact, and its compact elements are closed under finite meets. An element $c$ of a frame is \emph{complemented} if there exists $d$ with \(c\vee d=1\) and \(c\wedge d=0\). A frame is \emph{zero-dimensional} if its complemented elements are join-dense, equivalently, if every element is the join of the complemented elements below it. A \emph{coherent frame homomorphism} preserves compact elements. The resulting category is \(\CohFrm\). We write \(\AlgLat\) for algebraic lattices and Scott-continuous maps, and \(\DLat\) for bounded distributive lattices and bounded lattice homomorphisms.

A map \(j\colon L\to L\) on a frame is a \emph{nucleus} if it is extensive, idempotent, and preserves finite meets. Its fixed points form a frame \(L_j\), and the corestriction \(j\colon L\to L_j\) is a frame homomorphism. The corresponding sublocale is dense exactly when \(j(0)=0\).

We use the standard dual equivalence
\[
\SpecSp^{\op}\simeq\CohFrm,
\qquad
X\longmapsto\Omega(X),
\qquad
F\longmapsto\Sigma(F),
\]
and the equivalence
\[
\CohFrm\simeq\DLat,
\qquad
F\longmapsto\mathfrak k(F),
\qquad
L\longmapsto\mathsf{Idl}(L),
\]
where \(\mathfrak k(F)\) is the lattice of compact elements and \(\mathsf{Idl}(L)\) is the frame of lattice ideals; see \cite[Chapter~II]{Johnstone} and \cite[Chapter~II]{Picado}.

For later use, a frame homomorphism \(h\colon L\to M\) is called \emph{onto} if \(hh_*=1_M\), \emph{dense} if \(h(a)=0\) implies \(a=0\), and \emph{codense} if \(h(a)=1\) implies \(a=1\). A frame $L$ is \emph{joinfit} if for every \(a>0\) there exists \(b<1\) with \(a\vee b=1\).

\section{The radical-ideal functors and their comparison}\label{sec:radical-functors}

The constructions in this section are first defined objectwise and then promoted to functors by extension--contraction adjunctions. We begin with the coherent $k$-radical frame and its compact elements, pass to functorial extension maps, and finish with the natural nucleus comparing ordinary and $k$-radical ideals. This order separates the lattice-theoretic input from the genuinely categorical comparison.

\subsection{The \texorpdfstring{$k$}{k}-radical frame}

We first construct the $k$-radical object at a fixed semiring. The main goals are to identify its frame operations with the topology of the $k$-prime spectrum, determine its compact elements, and prove coherence before any morphisms are introduced.

For a family \((I_\lambda)_{\lambda\in\Lambda}\) in \(\RIdk(S)\), define
\[
\mathop{\bigvee_{\!k}}_{\lambda\in\Lambda}I_\lambda
 =\sqrt[k]{\sum_{\lambda\in\Lambda}I_\lambda},
\qquad
\bigwedge_{\lambda\in\Lambda}I_\lambda
 =\bigcap_{\lambda\in\Lambda}I_\lambda.
\]
Because $(0)$ is $k$-radical under the standing hypotheses, these operations make \(\RIdk(S)\) a complete lattice with bottom $(0)$ and top $S$.

\begin{remark}\label{rem:ordinary-radical-frame}
The radical ideals form a complete lattice \(\RId(S)\) with
\[
\bigvee_{\lambda\in\Lambda}I_\lambda
 =\sqrt{\sum_{\lambda\in\Lambda}I_\lambda},
\qquad
\bigwedge_{\lambda\in\Lambda}I_\lambda
 =\bigcap_{\lambda\in\Lambda}I_\lambda.
\]
Every $k$-radical ideal is radical, but the inclusion of underlying posets need not preserve joins. Thus \(\RIdk(S)\) should not in general be regarded as a subframe of \(\RId(S)\); the correct relationship is the nucleus constructed in Subsection~\ref{subsec:natural-comparison}.
\end{remark}

We shall use the following product formula.

\begin{lemma}[{\cite[Lemma~3.9]{AG24}}]\label{lem:k-radical-product}
For \(I,J\in\RIdk(S)\),
\[
\sqrt[k]{I\odot_k J}=I\cap J.
\]
\end{lemma}

\begin{proposition}\label{prop:k-radical-open-frame}
For every \(S\in\CRig\), the map
\[
\delta^k_S\colon\RIdk(S)\longrightarrow\Omega(\Speck(S)),
\qquad
\delta^k_S(I)=\bigcup_{a\in I}D_k(a)
              =\Speck(S)\setminus\mathcal V_k(I),
\]
is a frame isomorphism. In particular, \(\RIdk(S)\) is a frame.
\end{proposition}

\begin{proof}
An ideal $K$ and its $k$-radical are contained in the same $k$-prime ideals, hence
\[
\delta^k_S(\sqrt[k]{K})=\bigcup_{a\in K}D_k(a).
\]
If \(\delta^k_S(I)=\delta^k_S(J)\), then \(\mathcal V_k(I)=\mathcal V_k(J)\), and the $k$-radicality of $I$ and $J$ gives
\[
I=\bigcap_{P\in\mathcal V_k(I)}P
 =\bigcap_{P\in\mathcal V_k(J)}P=J.
\]
Thus \(\delta^k_S\) is injective. It is surjective because the sets \(D_k(a)\) form a basis: an open set \(\bigcup_{a\in A}D_k(a)\) is the image of \(\sqrt[k]{\langle A\rangle}\).

The definition of joins gives
\[
\delta^k_S\left(\mathop{\bigvee_{\!k}}_\lambda I_\lambda\right)
 =\bigcup_\lambda\delta^k_S(I_\lambda).
\]
For a $k$-prime ideal $P$,
\[
I\cap J\subseteq P
\quad\Longleftrightarrow\quad
I\subseteq P\text{ or }J\subseteq P.
\]
The reverse implication is immediate. Conversely, if neither $I$ nor $J$ is contained in $P$, choose $x\in I\setminus P$ and $y\in J\setminus P$. Then $xy\in I\cap J$, whereas primality gives $xy\notin P$, a contradiction. Consequently,
\begin{align*}
P\in\delta^k_S(I\cap J)
&\Longleftrightarrow I\cap J\nsubseteq P\\
&\Longleftrightarrow I\nsubseteq P\text{ and }J\nsubseteq P\\
&\Longleftrightarrow P\in\delta^k_S(I)\cap\delta^k_S(J).
\end{align*}
Thus \(\delta^k_S\) preserves finite meets and is a frame isomorphism.
\end{proof}

For a finite family \(a_1,\ldots,a_n\in S\), write
\[
[a_1,\ldots,a_n]_k
 =\sqrt[k]{\langle a_1,\ldots,a_n\rangle};
\]
in particular, \([a]_k=\sqrt[k]{\langle a\rangle}\).

\begin{lemma}\label{lem:finite-k-radical-join}
For every finite family \(a_1,\ldots,a_n\in S\),
\[
[a_1,\ldots,a_n]_k
 =[a_1]_k\mathbin{\bigvee_{\!k}}\cdots
  \mathbin{\bigvee_{\!k}}[a_n]_k.
\]
\end{lemma}

\begin{proof}
Both sides are the least $k$-radical ideal containing the given finite family.
\end{proof}

\begin{lemma}\label{lem:k-radical-compacts}
The compact elements of \(\RIdk(S)\) are precisely the ideals
\([a_1,\ldots,a_n]_k\) generated $k$-radically by a finite family.
\end{lemma}

\begin{proof}
Under \(\delta^k_S\), the principal $k$-radical ideal \([a]_k\) corresponds to the quasi-compact basic open \(D_k(a)\). Hence principal $k$-radical ideals are compact, and their finite joins are compact by Lemma~\ref{lem:finite-k-radical-join}.

Conversely, every \(I\in\RIdk(S)\) is the join of the principal $k$-radical ideals below it:
\[
I=\mathop{\bigvee_{\!k}}_{a\in I}[a]_k.
\]
If $I$ is compact, a finite subjoin already equals $I$, and Lemma~\ref{lem:finite-k-radical-join} yields the required form.
\end{proof}

\begin{lemma}\label{lem:k-radical-algebraic}
The frame \(\RIdk(S)\) is algebraic.
\end{lemma}

\begin{proof}
The equality \(I=\bigvee_{\!k,a\in I}[a]_k\) expresses every element as a join of compact elements below it.
\end{proof}

\begin{theorem}\label{thm:k-radical-coherent}
For every \(S\in\CRig\), the frame \(\RIdk(S)\) is coherent.
\end{theorem}

\begin{proof}
It is algebraic by Lemma~\ref{lem:k-radical-algebraic}, and its top element is the compact element \(S=[1]_k\). Moreover,
\[
\delta^k_S([ab]_k)=D_k(ab)=D_k(a)\cap D_k(b)
 =\delta^k_S([a]_k\cap[b]_k),
\]
so \([ab]_k=[a]_k\cap[b]_k\). If
\[
I=\mathop{\bigvee_{\!k}}_{i=1}^{m}[a_i]_k,
\qquad
J=\mathop{\bigvee_{\!k}}_{j=1}^{n}[b_j]_k
\]
are compact, distributivity gives
\[
I\cap J
 =\mathop{\bigvee_{\!k}}_{1\leq i\leq m,\,1\leq j\leq n}
  [a_i b_j]_k,
\]
which is compact. Thus compact elements are closed under finite meets.
\end{proof}

\subsection{Extension, contraction, and functoriality}

We now pass from the objectwise frames to their action on semiring homomorphisms. Contraction is the right adjoint, radical extension is its left adjoint, and this adjoint description supplies preservation properties and functoriality without elementwise choices.

Let \(f\colon S\to T\) be a morphism in \(\CRig\). Contraction is the inverse-image map
\[
f^{-1}\colon\RIdk(T)\longrightarrow\RIdk(S).
\]
This map is well defined: if \(J=\bigcap_{\lambda}P_\lambda\) is $k$-radical in $T$, then
\[
f^{-1}(J)=\bigcap_{\lambda}f^{-1}(P_\lambda),
\]
with the evident interpretation for an empty intersection, and every \(f^{-1}(P_\lambda)\) is $k$-prime by Proposition~\ref{prop:spectrum-functors}. Its left adjoint is the $k$-radical extension map defined below.

\begin{definition}
For \(I\in\RIdk(S)\), set
\[
\RIdk(f)(I)=\sqrt[k]{\langle f[I]\rangle}.
\]
\end{definition}

\begin{remark}\label{rem:k-extension-contraction}
The ideal \(\RIdk(f)(I)\) is the least $k$-radical ideal of $T$ containing \(f[I]\). For \(I\in\RIdk(S)\) and \(J\in\RIdk(T)\),
\begin{equation}\label{eq:k-extension-contraction}
\RIdk(f)(I)\subseteq J
\quad\Longleftrightarrow\quad
I\subseteq f^{-1}(J).
\end{equation}
Thus
\[
\RIdk(f)\dashv f^{-1}.
\]
This adjunction is the categorical reason for the definition of the direct-image map.
\end{remark}

\begin{lemma}\label{lem:k-extension-principal}
For every \(x\in S\),
\[
\RIdk(f)([x]_k)=[f(x)]_k.
\]
\end{lemma}

\begin{proof}
For every \(J\in\RIdk(T)\), the adjunction \eqref{eq:k-extension-contraction} gives
\begin{align*}
\RIdk(f)([x]_k)\subseteq J
&\Longleftrightarrow [x]_k\subseteq f^{-1}(J)\\
&\Longleftrightarrow x\in f^{-1}(J)\\
&\Longleftrightarrow f(x)\in J\\
&\Longleftrightarrow [f(x)]_k\subseteq J.
\end{align*}
The two $k$-radical ideals have the same upper bounds and are therefore equal.
\end{proof}

\begin{proposition}\label{prop:k-extension-coherent}
For every morphism \(f\colon S\to T\) in \(\CRig\), the map
\[
\RIdk(f)\colon\RIdk(S)\longrightarrow\RIdk(T)
\]
is a coherent frame homomorphism.
\end{proposition}

\begin{proof}
As a left adjoint, \(\RIdk(f)\) preserves arbitrary joins. It preserves bottom because \(f(0)=0\) and $(0)$ is $k$-radical, and it preserves top because \(f(1)=1\).

Let \(I,J\in\RIdk(S)\). For every \(P\in\Speck(T)\), the contraction \(f^{-1}(P)\) is $k$-prime. For ideals $A,B$ and a prime ideal $Q$, one has
\[
A\cap B\subseteq Q
\quad\Longleftrightarrow\quad
A\subseteq Q\text{ or }B\subseteq Q:
\]
if neither inclusion holds, choose $a\in A\setminus Q$ and $b\in B\setminus Q$; then $ab\in A\cap B$ but $ab\notin Q$. Therefore
\begin{align*}
\RIdk(f)(I\cap J)\subseteq P
&\Longleftrightarrow I\cap J\subseteq f^{-1}(P)\\
&\Longleftrightarrow I\subseteq f^{-1}(P)\ \text{or}\ J\subseteq f^{-1}(P)\\
&\Longleftrightarrow \RIdk(f)(I)\subseteq P\ \text{or}\ \RIdk(f)(J)\subseteq P\\
&\Longleftrightarrow \RIdk(f)(I)\cap\RIdk(f)(J)\subseteq P.
\end{align*}
The two sides are $k$-radical ideals lying in exactly the same $k$-prime ideals, hence they are equal. Thus finite meets are preserved. Finally, Lemmas~\ref{lem:k-radical-compacts} and~\ref{lem:k-extension-principal} show that compact elements are sent to compact elements.
\end{proof}

\begin{proposition}\label{prop:k-radical-functor}
The assignments
\[
S\longmapsto\RIdk(S),
\qquad
(f\colon S\to T)\longmapsto\RIdk(f)
\]
define a covariant functor
\[
\RIdk\colon\CRig\longrightarrow\CohFrm.
\]
\end{proposition}

\begin{proof}
The right adjoint of \(\RIdk(f)\) is contraction. The identity homomorphism has identity contraction, so uniqueness of left adjoints gives \(\RIdk(1_S)=1_{\RIdk(S)}\). For composable maps \(S\xrightarrow{f}T\xrightarrow{g}U\),
\[
(gf)^{-1}=f^{-1}g^{-1}.
\]
Both \(\RIdk(gf)\) and \(\RIdk(g)\RIdk(f)\) are left adjoint to this composite contraction. Uniqueness of left adjoints yields
\[
\RIdk(gf)=\RIdk(g)\RIdk(f).
\]
\end{proof}

\begin{remark}\label{rem:coherent-frames-algebraic-lattices}
The category \(\CohFrm\) is a non-full subcategory of \(\AlgLat\): coherent frame homomorphisms are Scott-continuous, but not every Scott-continuous map is a frame homomorphism. Consequently, \(\RIdk\) may also be viewed as an algebraic-lattice-valued functor.
\end{remark}

\begin{lemma}\label{lem:k-extension-reflects-zero}
If \(f\colon S\to T\) is injective, then \(\RIdk(f)\) reflects the bottom element:
\[
\RIdk(f)(I)=(0)\quad\Longrightarrow\quad I=(0).
\]
\end{lemma}

\begin{proof}
Since \(f[I]\subseteq\RIdk(f)(I)\), the hypothesis implies \(f[I]=\{0\}\). Injectivity gives \(I=(0)\).
\end{proof}

\begin{proposition}\label{prop:k-extension-properties}
Let \(f\colon S\to T\) be a morphism in \(\CRig\).
\begin{enumerate}
\item If $f$ is surjective, then \(\RIdk(f)\) is onto as a frame homomorphism.
\item If $f$ is injective, then \(\RIdk(f)\) is dense.
\item If $f$ is an isomorphism, then \(\RIdk(f)\) is codense.
\item The frame \(\RIdk(S)\) is joinfit if and only if
\[
\bigcap_{M\in\Maxk(S)}M=(0).
\]
\end{enumerate}
\end{proposition}

\begin{proof}
(1) If \(J\in\RIdk(T)\) and $f$ is surjective, then \(f[f^{-1}(J)]=J\), and hence
\[
\RIdk(f)(f^{-1}(J))=\sqrt[k]{J}=J.
\]
Thus \(\RIdk(f)f^{-1}=1\), which is the stated onto condition.

(2) This is Lemma~\ref{lem:k-extension-reflects-zero}.

(3) If $f$ is an isomorphism, functoriality makes \(\RIdk(f)\) a frame isomorphism. It therefore reflects the top element.

(4) Suppose first that \(\bigcap\Maxk(S)=(0)\), and let \((0)<I\in\RIdk(S)\). Choose \(0\neq a\in I\) and a maximal $k$-ideal $M$ with \(a\notin M\). Maximal $k$-ideals are $k$-prime \cite{Sen}, hence $k$-radical. Since \(I\bigvee_k M\) properly contains $M$, maximality gives \(I\bigvee_k M=S\). Thus the frame is joinfit.

Conversely, assume that \(\RIdk(S)\) is joinfit and that \(0\neq a\) lies in every maximal $k$-ideal. Choose a proper \(J\in\RIdk(S)\) with \([a]_k\bigvee_k J=S\), and extend $J$ to a maximal $k$-ideal $M$. Then \(a\in M\), so \([a]_k\subseteq M\), while \(J\subseteq M\). This contradicts \([a]_k\bigvee_k J=S\). Therefore the intersection of the maximal $k$-ideals is zero.
\end{proof}

The ordinary radical construction is parallel. We include the statement and enough of the proof to make the functorial comparison explicit.

\begin{theorem}\label{thm:radical-functor}
For every \(S\in\CRig\) the following hold.
\begin{enumerate}
\item The compact elements of \(\RId(S)\) are precisely the ideals \([a_1,\ldots,a_n]\) generated radically by a finite family.
\item \(\RId(S)\) is a coherent frame.
\item For a morphism \(f\colon S\to T\) in \(\CRig\), the formula
\[
\RId(f)(I)=\sqrt{\langle f[I]\rangle}
\]
defines a coherent frame homomorphism, left adjoint to contraction, and
\(\RId(f)([x])=[f(x)]\).
\item These assignments define a covariant functor
\[
\RId\colon\CRig\longrightarrow\CohFrm.
\]
\end{enumerate}
\end{theorem}

\begin{proof}
Define
\[
\delta_S\colon\RId(S)\longrightarrow\Omega(\Spec(S)),
\qquad
\delta_S(I)=\bigcup_{a\in I}D(a).
\]
The proof of Proposition~\ref{prop:k-radical-open-frame}, with prime ideals and ordinary radicals in place of $k$-prime ideals and $k$-radicals, shows that \(\delta_S\) is a frame isomorphism. The compact-open basis \(\{D(a)\}\) then gives (1) and (2).

For (3), contraction sends radical ideals to radical ideals, and radical extension is its left adjoint. The argument of Proposition~\ref{prop:k-extension-coherent} proves preservation of finite meets; the adjunction gives arbitrary joins, and principal radical ideals are carried to principal radical ideals. Functoriality in (4) follows from uniqueness of left adjoints exactly as in Proposition~\ref{prop:k-radical-functor}.
\end{proof}

\begin{proposition}\label{prop:natural-spectral-representations}
The frame isomorphisms \(\delta_S\) and \(\delta^k_S\) are natural in $S$. Equivalently, there are natural isomorphisms
\[
\RId\cong\Omega\circ\Spec,
\qquad
\RIdk\cong\Omega\circ\Speck,
\]
where the spectrum functors are viewed as functors \(\CRig\to\SpecSp^{\op}\).
\end{proposition}

\begin{proof}
Let \(f\colon S\to T\) be a morphism in \(\CRig\). Both routes around the first naturality square preserve arbitrary joins, so it suffices to evaluate them on \([a]\):
\[
\delta_T(\RId(f)([a]))
 =\delta_T([f(a)])
 =D_T(f(a))
 =(f^*)^{-1}(D_S(a)).
\]
The same calculation with \([a]_k\) and \(D_k(a)\) proves naturality for \(\delta^k\).
\end{proof}

\subsection{The natural nucleus comparison}\label{subsec:natural-comparison}

The two radical doctrines are related by $k$-radicalization. In this subsection we show that this comparison is a nucleus, identify its fixed-point frame, prove naturality with respect to radical extension, and translate the resulting quotient into a dense spectral and localic embedding.

\begin{proposition}\label{prop:k-radical-nucleus}
For every \(S\in\CRig\), the map
\[
\mathfrak j_S\colon\RId(S)\longrightarrow\RId(S),
\qquad
\mathfrak j_S(I)=\sqrt[k]{I}
 =\mathop{\bigvee_{\!k}}_{a\in I}[a]_k,
\]
is a nucleus. Its fixed-point frame is \(\RIdk(S)\).
\end{proposition}

\begin{proof}
The map is monotone, extensive, and idempotent because $k$-radicalization is a closure operator. Its values are radical ideals, being intersections of prime ideals.

Let \(I,J\in\RId(S)\). For \(P\in\Speck(S)\),
\[
I\cap J\subseteq P
\quad\Longleftrightarrow\quad
I\cdot J\subseteq P
\quad\Longleftrightarrow\quad
I\subseteq P\ \text{or}\ J\subseteq P.
\]
For the first equivalence, if \(I\cdot J\subseteq P\) and \(x\in I\cap J\), then \(x^2\in P\), whence \(x\in P\) because $P$ is radical. The second equivalence is primality. Thus
\(\mathcal V_k(I\cap J)=\mathcal V_k(I)\cup\mathcal V_k(J)\), and therefore
\[
\mathfrak j_S(I\cap J)=\mathfrak j_S(I)\cap\mathfrak j_S(J).
\]
The fixed points are exactly the $k$-radical ideals by definition.
\end{proof}

We use the same symbol \(\mathfrak j_S\) for the corestriction
\(\RId(S)\to\RIdk(S)\).

\begin{corollary}\label{cor:nucleus-coherent-map}
The map \(\mathfrak j_S\colon\RId(S)\to\RIdk(S)\) is a coherent frame homomorphism.
\end{corollary}

\begin{proof}
A nucleus corestriction is a frame homomorphism \cite[Chapter~II, Lemma~2.2]{Johnstone}. By Theorem~\ref{thm:radical-functor}, the compact elements of \(\RId(S)\) are the finite radical ideals. For such a family,
\[
\mathfrak j_S([a_1,\ldots,a_n])=[a_1,\ldots,a_n]_k,
\]
because both sides are the intersection of the $k$-prime ideals containing the generators. The image is compact by Lemma~\ref{lem:k-radical-compacts}.
\end{proof}

\begin{theorem}\label{thm:k-radical-dense-sublocale}
For every \(S\in\CRig\), the nucleus \(\mathfrak j_S\) presents the locale with frame \(\RIdk(S)\) as a dense sublocale of the locale with frame \(\RId(S)\).
\end{theorem}

\begin{proof}
The sublocale is determined by the nucleus \(\mathfrak j_S\). Since $(0)$ is $k$-radical under the standing assumptions,
\(\mathfrak j_S(0)=0\); this is precisely the density criterion for a nucleus.
\end{proof}

\begin{proposition}\label{prop:natural-k-radical-comparison}
The family \((\mathfrak j_S)_{S\in\CRig}\) defines a natural transformation
\[
\mathfrak j\colon\RId\Longrightarrow\RIdk.
\]
Moreover, under the natural spectral representations of Proposition~\ref{prop:natural-spectral-representations}, the component \(\mathfrak j_S\) is restriction of opens along
\(\iota_S\colon\Speck(S)\hookrightarrow\Spec(S)\); explicitly, the square
\[
\begin{tikzcd}[column sep=large]
\RId(S) \arrow[r,"\mathfrak j_S"] \arrow[d,"\delta_S"']
  & \RIdk(S) \arrow[d,"\delta^k_S"] \\
\Omega(\Spec(S)) \arrow[r,"\iota_S^{-1}"']
  & \Omega(\Speck(S))
\end{tikzcd}
\]
commutes.
\end{proposition}

\begin{proof}
Let \(f\colon S\to T\). On a principal radical ideal,
\begin{align*}
(\RIdk(f)\mathfrak j_S)([a])
 &=\RIdk(f)([a]_k)=[f(a)]_k\\
 &=\mathfrak j_T([f(a)])
 =(\mathfrak j_T\RId(f))([a]).
\end{align*}
Every radical ideal is a join of principal radical ideals, and both composites preserve joins. Hence the naturality square commutes.

For the spectral square, both routes preserve arbitrary joins, and on \([a]\) they give
\[
D(a)\cap\Speck(S)=D_k(a).
\]
This proves the assertion.
\end{proof}

\begin{corollary}\label{cor:k-spectrum-dense}
Under the standing reduced and conical hypotheses, the natural spectral embedding
\[
\iota_S\colon\Speck(S)\hookrightarrow\Spec(S)
\]
has dense image.
\end{corollary}

\begin{proof}
The frame map \(\iota_S^{-1}\) is identified with the dense frame homomorphism \(\mathfrak j_S\). Since both locales are spatial, locale density is equivalent to topological density of the corresponding spectral embedding.
\end{proof}

\section{Support functors and categorical representations}\label{sec:support-functors}

The preceding section represents radical ideals by open sets. We now pass to the compact part of this representation and formulate support as a universal lattice-valued invariant. After proving reconstruction and its $k$- and strong variants, we compare all three spectra at the level of order, constructible topology, locales, and sheaves. We then isolate a density criterion for $r$-semirings, derive the Stone criterion, establish the fixed-base universal properties, and construct the global support bifibration. The relevant chain of dualities is
\[
\xymatrix@C-=1pc@R=2pc{
\CRig \ar@{->}[rrr]^{\Spec} & & &
\SpecSp^{\op}
\ar@/^1pc/@{->}[rrrr]^{\Omega} & & \simeq & &
\CohFrm
\ar@/^1pc/@{->}[llll]^{\Sigma}
\ar@/^1pc/@{->}[rrrr]^{\mathfrak k(-)} & & \simeq & &
\DLat
\ar@/^1pc/@{->}[llll]^{\mathsf{Idl}(-)}.
}
\]

\begin{remark}\label{rem:stone-ideal-completion}
The first equivalence is coherent Stone duality between spectral spaces and coherent frames, while the second is ideal completion: a coherent frame is recovered from its compact elements, and a bounded distributive lattice is recovered as the compact part of its ideal frame. See \cite[Chapter~II]{Johnstone,Picado}. Consequently, a semiring $S$ determines, up to canonical isomorphism, a bounded distributive lattice whose prime spectrum is homeomorphic to \(\Spec(S)\).
\end{remark}

Joyal's support for rings is an early form of this construction \cite{JoyalSupport}; see also \cite[Chapter~12]{SpectralBook} and \cite[Chapter~V, Section~3]{Johnstone}. We use a definition that records both the algebraic relations and the generating property needed for a universal representation.

\subsection{Support objects and their basic properties}

We begin with the intrinsic definition of a support and then extract the algebraic information forced by its axioms. In particular, we establish the identities used later in reconstruction and show that complemented idempotents are represented exactly by complemented lattice elements.

\begin{definition}\label{def:support}
Let \(S\in\CRig\) and let $L$ be a bounded distributive lattice. A map
\(\mathfrak D\colon S\to L\) is a \emph{support} if the image of \(\mathfrak D\) generates $L$ as a bounded distributive lattice and, for all \(a,b,b_1,\ldots,b_n\in S\),
\begin{enumerate}
\item \(\mathfrak D(0)=0\) and \(\mathfrak D(1)=1\);
\item \(\mathfrak D(a+b)\leq\mathfrak D(a)\vee\mathfrak D(b)\);
\item \(\mathfrak D(ab)=\mathfrak D(a)\wedge\mathfrak D(b)\);
\item if
\(\mathfrak D(a)\leq\mathfrak D(b_1)\vee\cdots\vee\mathfrak D(b_n)\), then
\(a\in[b_1,\ldots,b_n]\).
\end{enumerate}
A \emph{frame support} is defined similarly, with a frame $F$ generated by the image under arbitrary joins and finite meets.

A morphism of supports
\((L,\mathfrak D)\to(M,\mathfrak E)\) is a bounded lattice homomorphism
\(h\colon L\to M\) satisfying \(h\mathfrak D=\mathfrak E\). Frame supports and frame homomorphisms define the analogous category. We denote these categories by \(\Supp(S)\) and \(\FSupp(S)\), respectively.
\end{definition}

The canonical compact-open support is
\[
D_S\colon S\longrightarrow\KOpen(\Spec(S)),
\qquad
D_S(a)=D(a),
\]
and the canonical frame support is
\[
\rho_S\colon S\longrightarrow\RId(S),
\qquad
\rho_S(a)=[a].
\]
Theorem~\ref{thm:basic-open-cover} verifies the reflection axiom in both cases.

\begin{proposition}\label{prop:support-properties}
Let \((L,\mathfrak D)\) be a support or a frame support of $S$. Then:
\begin{enumerate}
\item if $x$ is a unit, then \(\mathfrak D(x)=1\);
\item \(\mathfrak D(x^n)=\mathfrak D(x)\) for every integer \(n\geq1\);
\item if $x$ is nilpotent, then \(\mathfrak D(x)=0\);
\item if \(xy=0\), then \(\mathfrak D(x+y)=\mathfrak D(x)\vee\mathfrak D(y)\);
\item if \(x\in[y_1,\ldots,y_n]\), then
\[
\mathfrak D(x)\leq\mathfrak D(y_1)\vee\cdots\vee\mathfrak D(y_n).
\]
\end{enumerate}
\end{proposition}

\begin{proof}
If \(xy=1\), then
\(1=\mathfrak D(x)\wedge\mathfrak D(y)\), proving (1). Repeated application of multiplicativity gives (2), and (3) follows by applying (2) to a power equal to zero.

For (4), subadditivity gives one inequality. Since \(xy=0\),
\[
\mathfrak D(x)\wedge\mathfrak D(x+y)
 =\mathfrak D(x^2+xy)
 =\mathfrak D(x^2)
 =\mathfrak D(x),
\]
so \(\mathfrak D(x)\leq\mathfrak D(x+y)\); the argument for $y$ is symmetric.

For (5), choose \(m\geq1\) and coefficients \(s_i\in S\) such that
\(x^m=\sum_i s_i y_i\). Then
\[
\mathfrak D(x)=\mathfrak D(x^m)
 \leq\bigvee_i\mathfrak D(s_i y_i)
 \leq\bigvee_i\mathfrak D(y_i).
\]
\end{proof}

For a bounded distributive lattice $L$, let \(\mathfrak B(L)\) be its Boolean algebra of complemented elements. With
\[
x\oplus y=(x\wedge y^\bot)\vee(x^\bot\wedge y),
\qquad
x\odot y=x\wedge y,
\]
it is a Boolean ring. An idempotent \(e\in S\) is \emph{complemented} if there exists \(e^\perp\in S\) with
\(e+e^\perp=1\) and \(ee^\perp=0\). The complement is unique. Let \(\operatorname{Comp}(S)\) denote the complemented idempotents, equipped with
\[
e\boxplus f=ef^\perp+fe^\perp,
\qquad
e\boxdot f=ef.
\]
This is a Boolean ring \cite[Proposition~2.1.3]{Chermnykh12}.

\begin{proposition}\label{prop:support-idempotents}
Every support \(\mathfrak D\colon S\to L\) restricts to a Boolean-ring isomorphism
\[
\operatorname{Comp}(S)\xrightarrow{\ \cong\ }\mathfrak B(L).
\]
In particular, the Boolean algebra of complemented idempotents is determined by the support object.
\end{proposition}

\begin{proof}
If \(e\in\operatorname{Comp}(S)\), Proposition~\ref{prop:support-properties}(4) gives
\[
\mathfrak D(e)\vee\mathfrak D(e^\perp)=1,
\qquad
\mathfrak D(e)\wedge\mathfrak D(e^\perp)=0,
\]
so \(\mathfrak D(e^\perp)=\mathfrak D(e)^\bot\).

Suppose \(\mathfrak D(e)=\mathfrak D(f)\). Then
\[
\mathfrak D(ef^\perp)
 =\mathfrak D(e)\wedge\mathfrak D(f)^\bot=0.
\]
Since \(\mathfrak D(ef^\perp)=0=\mathfrak D(0)\), the reflection axiom gives
\(ef^\perp\in[0]=\sqrt{(0)}\). Reducedness of $S$ therefore yields
\(ef^\perp=0\). Hence
\(e=e(f+f^\perp)=ef\); symmetrically, \(f=ef\), so \(e=f\). Thus the restriction is injective.

Let \(c\in\mathfrak B(L)\). Because $L$ is generated by the support values and finite meets of support values are again support values, there are finite families \(a_i,b_j\in S\) with
\[
c=\bigvee_i\mathfrak D(a_i),
\qquad
c^\bot=\bigvee_j\mathfrak D(b_j).
\]
The equality \(c\wedge c^\bot=0\) implies \(\mathfrak D(a_i b_j)=0\). The reflection axiom gives \(a_i b_j\in[0]=\sqrt{(0)}\), and reducedness therefore yields \(a_i b_j=0\). Since
\[
\mathfrak D(1)\leq
\bigvee_i\mathfrak D(a_i)\vee\bigvee_j\mathfrak D(b_j),
\]
the reflection axiom gives an equality
\[
1=\sum_i r_i a_i+\sum_j s_j b_j
\]
for suitable coefficients. Put \(e=\sum_i r_i a_i\) and \(e^\perp=\sum_j s_j b_j\). Then \(e+e^\perp=1\) and \(ee^\perp=0\), so both are complementary idempotents. Moreover,
\(\mathfrak D(e)\leq c\), \(\mathfrak D(e^\perp)\leq c^\bot\), and their join is $1$; distributivity gives \(\mathfrak D(e)=c\). Thus the restriction is surjective.

Finally,
\[
(ef^\perp)(fe^\perp)=0.
\]
Proposition~\ref{prop:support-properties}(4), multiplicativity, and preservation of complements therefore give
\begin{align*}
\mathfrak D(e\boxplus f)
 &=(\mathfrak D(e)\wedge\mathfrak D(f)^\bot)
  \vee(\mathfrak D(f)\wedge\mathfrak D(e)^\bot),\\
\mathfrak D(e\boxdot f)
 &=\mathfrak D(e)\wedge\mathfrak D(f),
\end{align*}
so the bijection is a Boolean-ring isomorphism.
\end{proof}

Recall that a lattice-ordered ring $R$ is an \emph{$f$-ring} if
\((a\wedge b)c=ac\wedge bc\) for all \(a,b\in R\) and \(c\geq0\). Its positive cone \(R^+\) is a cancellative conical semiring \cite{Fuchs}.

\begin{lemma}\label{lem:f-ring-idempotents}
Let $R$ be a commutative unital reduced $f$-ring with \(1\geq0\), and put \(S=R^+\). Every idempotent \(e\in S\) satisfies \(0\leq e\leq1\); consequently, it is complemented in $S$, with complement \(1-e\).
\end{lemma}

\begin{proof}
Let \(u=(e-1)^+\). Since multiplication by the positive element $e$ preserves finite meets and hence finite joins and positive parts in an $f$-ring,
\[
eu=e(e-1)^+=(e(e-1))^+=0.
\]
Moreover, \(0\leq u\leq e\), and multiplication by $u\geq0$ gives
\[
0\leq u^2\leq eu=0.
\]
Reducedness implies $u=0$, so $e\leq1$. Hence \(1-e\in S\), and
\[
e+(1-e)=1,\qquad e(1-e)=0.
\]
Thus $1-e$ is the complement of $e$.
\end{proof}

\subsection{Spectral and frame reconstruction}

We now prove that a support is not merely an auxiliary invariant: it reconstructs both the prime spectrum and the radical-ideal frame. The resulting identifications are then assembled naturally for the canonical compact-open support functor.

A bounded distributive lattice $L$ is a semiring under \(\vee\) and \(\wedge\). Its semiring ideals are its lattice ideals, every such ideal is strong, and
\(\Spec(L)=\Speck(L)\).

\begin{proposition}\label{prop:support-spectrum-homeomorphism}
A support \(\mathfrak D\colon S\to L\) induces a homeomorphism
\[
\varphi_{\mathfrak D}\colon\Spec(L)\longrightarrow\Spec(S),
\qquad
\varphi_{\mathfrak D}(\mathfrak p)=\mathfrak D^{-1}(\mathfrak p).
\]
Its inverse is
\[
\psi_{\mathfrak D}(\mathfrak q)
 =\bigl\langle\mathfrak D(a)\mid a\in\mathfrak q\bigr\rangle.
\]
\end{proposition}

\begin{proof}
If \(\mathfrak p\in\Spec(L)\), the preimage \(\mathfrak D^{-1}(\mathfrak p)\) is a proper ideal: closure under addition follows from subadditivity, closure under multiplication from multiplicativity, and properness from \(\mathfrak D(1)=1\). It is prime because
\(\mathfrak D(ab)=\mathfrak D(a)\wedge\mathfrak D(b)\in\mathfrak p\) forces one of the two factors into \(\mathfrak p\).

Conversely, let \(\mathfrak q\in\Spec(S)\) and put
\(\mathfrak p_{\mathfrak q}=\langle\mathfrak D(\mathfrak q)\rangle\). It is proper: if
\(1\leq\bigvee_i\mathfrak D(q_i)\) for \(q_i\in\mathfrak q\), then the reflection axiom gives
\(1\in[q_1,\ldots,q_n]\subseteq\mathfrak q\), a contradiction. Because the image of \(\mathfrak D\) generates $L$ as a bounded distributive lattice and finite meets of support values satisfy
\[
\mathfrak D(a_1)\wedge\cdots\wedge\mathfrak D(a_r)
 =\mathfrak D(a_1\cdots a_r),
\]
every element of $L$ is a finite join of support values. For
\(x=\bigvee_i\mathfrak D(a_i)\) one has
\begin{equation}\label{eq:support-membership-prime}
x\in\mathfrak p_{\mathfrak q}
\quad\Longleftrightarrow\quad
a_i\in\mathfrak q\text{ for every }i.
\end{equation}
Indeed, if \(x\in\mathfrak p_{\mathfrak q}\), then
\[
x\leq\bigvee_{j=1}^m\mathfrak D(q_j)
\]
for some \(q_j\in\mathfrak q\). Hence each \(\mathfrak D(a_i)\) lies below that finite join, and the reflection axiom gives
\(a_i\in[q_1,\ldots,q_m]\subseteq\mathfrak q\). The reverse implication is immediate.

Now suppose \(x\wedge y\in\mathfrak p_{\mathfrak q}\) but \(x\notin\mathfrak p_{\mathfrak q}\). Write
\(x=\bigvee_i\mathfrak D(a_i)\) and \(y=\bigvee_j\mathfrak D(b_j)\). By \eqref{eq:support-membership-prime}, some \(a_{i_0}\notin\mathfrak q\), whereas
\(a_i b_j\in\mathfrak q\) for every pair \((i,j)\). Primality gives \(b_j\in\mathfrak q\) for every $j$, and hence \(y\in\mathfrak p_{\mathfrak q}\). Thus \(\mathfrak p_{\mathfrak q}\) is prime.

Equation~\eqref{eq:support-membership-prime}, applied to a single support value, gives
\(\mathfrak D^{-1}(\mathfrak p_{\mathfrak q})=\mathfrak q\). Conversely, let \(\mathfrak p\in\Spec(L)\). The lattice ideal generated by
\(\mathfrak D[\mathfrak D^{-1}(\mathfrak p)]\) is contained in \(\mathfrak p\). For the reverse inclusion, if
\(x=\bigvee_i\mathfrak D(a_i)\in\mathfrak p\), then each \(\mathfrak D(a_i)\leq x\) lies in \(\mathfrak p\), so every $a_i$ belongs to \(\mathfrak D^{-1}(\mathfrak p)\). Hence
\(x\in\psi_{\mathfrak D}\varphi_{\mathfrak D}(\mathfrak p)\). The maps are therefore inverse. Finally,
\[
\varphi_{\mathfrak D}^{-1}(D_S(a))=D_L(\mathfrak D(a)).
\]
These sets form a basis of \(\Spec(L)\): if \(x=\bigvee_i\mathfrak D(a_i)\), then
\[
D_L(x)=\bigcup_i D_L(\mathfrak D(a_i)).
\]
Thus the bijection is a homeomorphism.
\end{proof}

\begin{corollary}\label{cor:support-category-contractible}
For a fixed semiring $S$, the category \(\Supp(S)\) is equivalent to the terminal category. Equivalently, any two supports of $S$ are connected by a unique isomorphism over $S$.
\end{corollary}

\begin{proof}
Let \((L,\mathfrak D)\) be a support. Bounded distributive lattice duality identifies $L$ with \(\KOpen(\Spec(L))\), and Proposition~\ref{prop:support-spectrum-homeomorphism} transports this lattice to \(\KOpen(\Spec(S))\). Under the resulting isomorphism,
\(\mathfrak D(a)\) corresponds to \(D_S(a)\), because
\[
\varphi_{\mathfrak D}^{-1}(D_S(a))=D_L(\mathfrak D(a)).
\]
Thus every support is isomorphic over $S$ to the canonical compact-open support. Such an isomorphism is unique because the support values generate the codomain lattice. Hence \(\Supp(S)\) is a connected thin groupoid, and therefore is equivalent to the terminal category.
\end{proof}

\begin{remark}\label{rem:canonical-support-functor}
For a homomorphism \(f\colon S\to T\), inverse image along
\(f^*\colon\Spec(T)\to\Spec(S)\) restricts to a bounded lattice homomorphism
\[
\mathfrak D(f)\colon\KOpen(\Spec(S))\longrightarrow\KOpen(\Spec(T)),
\qquad
D_S(a)\longmapsto D_T(f(a)).
\]
Hence
\[
\mathfrak D=\KOpen\circ\Spec\colon\CRig\longrightarrow\DLat
\]
is a covariant functor, where \(\Spec\) is viewed as taking values in \(\SpecSp^{\op}\).
\end{remark}

\begin{definition}\label{def:canonical-support-functor}
The functor \(\mathfrak D\colon\CRig\to\DLat\) of Remark~\ref{rem:canonical-support-functor} is called the \emph{canonical support functor}.
\end{definition}

\begin{proposition}\label{prop:support-frame-reconstruction}
If \((L,\mathfrak D)\) is a support of $S$, then there is a frame isomorphism
\[
\Id(L)\cong\RId(S).
\]
The isomorphism is given by inverse image under \(\mathfrak D\), with inverse generated direct image.
\end{proposition}

\begin{proof}
Define
\[
\Phi\colon\Id(L)\longrightarrow\RId(S),
\qquad
\Phi(J)=\mathfrak D^{-1}(J),
\]
and
\[
\Psi\colon\RId(S)\longrightarrow\Id(L),
\qquad
\Psi(I)=\langle\mathfrak D(a)\mid a\in I\rangle.
\]
The preimage \(\Phi(J)\) is an ideal by the first three support axioms. It is radical because
\(x^n\in\Phi(J)\) implies
\(\mathfrak D(x)=\mathfrak D(x^n)\in J\).

Clearly \(\Psi\Phi(J)\subseteq J\). Conversely, if
\(x=\bigvee_i\mathfrak D(a_i)\in J\), then each \(\mathfrak D(a_i)\leq x\) lies in $J$, so \(a_i\in\Phi(J)\) and \(x\in\Psi\Phi(J)\). Hence \(\Psi\Phi=1\).

For an arbitrary ideal $I$ of $S$,
\begin{equation}\label{eq:support-radical-closure}
\Phi\Psi(I)=\sqrt I.
\end{equation}
Indeed, if \(x\in\Phi\Psi(I)\), then
\(\mathfrak D(x)\leq\bigvee_i\mathfrak D(a_i)\) for finitely many \(a_i\in I\), and the reflection axiom gives
\(x\in[a_1,\ldots,a_n]\subseteq\sqrt I\). Conversely, if \(x^n\in I\), then
\(\mathfrak D(x)=\mathfrak D(x^n)\in\Psi(I)\). Thus \(\Phi\Psi(I)=I\) for radical $I$. The inverse order isomorphisms preserve the complete lattice operations and hence are frame isomorphisms.
\end{proof}

\begin{remark}\label{rem:natural-support-reconstruction}
For the canonical support functor, Proposition~\ref{prop:support-frame-reconstruction} is natural in $S$ and gives
\[
\mathsf{Idl}\circ\mathfrak D\cong\RId.
\]
Indeed, for a morphism \(f\colon S\to T\), the two induced frame maps preserve arbitrary joins and agree on the principal compact generators: generated direct image carries the lattice ideal generated by \(D_S(a)\) to the lattice ideal generated by \(D_T(f(a))\), while
\(\RId(f)([a])=[f(a)]\). Together with Proposition~\ref{prop:natural-spectral-representations}, this yields natural identifications
\[
\Omega(\Spec(S))\cong\RId(S),
\qquad
\Omega(\Speck(S))\cong\RIdk(S),
\]
and therefore
\[
\Sigma\RId(S)\cong\Spec(S),
\qquad
\Sigma\RIdk(S)\cong\Speck(S).
\]
\end{remark}

\subsection{Variants: \texorpdfstring{$k$}{k}-supports and strong supports}

This subsection adapts the support formalism to the $k$-radical and strong-radical closure doctrines. We give canonical and function-semiring examples, construct the strong-prime spectrum functor, and prove the corresponding spectral representation for $\mathfrak l$-supports.

For \(b_1,\ldots,b_n\in S\), let
\([b_1,\ldots,b_n]_{\mathfrak l}\) be the least strong radical ideal containing the indicated elements. It exists as the intersection of all strong radical ideals containing them.

\begin{definition}\label{def:k-l-supports}
Let $L$ be a bounded distributive lattice generated by the image of a map
\(\mathfrak D\colon S\to L\) satisfying conditions (1)--(3) of Definition~\ref{def:support}.
\begin{enumerate}
\item The pair \((L,\mathfrak D)\) is a \emph{$k$-support} if
\[
\mathfrak D(a)\leq\mathfrak D(b_1)\vee\cdots\vee\mathfrak D(b_n)
\quad\Longrightarrow\quad
a\in[b_1,\ldots,b_n]_k.
\]
\item It is an \emph{$\mathfrak l$-support} if
\(\mathfrak D(a+b)=\mathfrak D(a)\vee\mathfrak D(b)\) and the same implication holds with
\([b_1,\ldots,b_n]_{\mathfrak l}\) in place of the $k$-radical ideal.
\end{enumerate}
Frame $k$-supports and frame $\mathfrak l$-supports are defined by replacing the bounded lattice with a frame generated under arbitrary joins and finite meets.
\end{definition}

\begin{example}\label{ex:canonical-k-support}
The map
\[
S\longrightarrow\KOpen(\Speck(S)),
\qquad
a\longmapsto D_k(a),
\]
is a $k$-support. Further examples are as follows.
\begin{enumerate}
\item Let \(\RIdk^{\mathrm{fin}}(S)\) denote the bounded distributive lattice of ideals \([a_1,\ldots,a_n]_k\). Then
\(a\mapsto[a]_k\) is a $k$-support on \(\RIdk^{\mathrm{fin}}(S)\) and a frame $k$-support on \(\RIdk(S)\).

\item Let $X$ be a Tychonoff space and let \(C^+(X)\) be the semiring of nonnegative continuous real-valued functions. The cozero map
\[
\Coz\colon C^+(X)\longrightarrow\Omega(X),
\qquad
\Coz(f)=\{x\in X\mid f(x)\neq0\},
\]
satisfies
\[
\Coz(f+g)=\Coz(f)\cup\Coz(g),
\qquad
\Coz(fg)=\Coz(f)\cap\Coz(g),
\]
and its image is a basis. It is not a frame $\mathfrak l$-support for every Tychonoff space. Take
\(X=[0,1]\), \(f(x)=x\), and
\[
g(0)=0,
\qquad
g(x)=e^{-1/x^2}\quad(x>0).
\]
Then \(\Coz(f)=\Coz(g)\). Every prime ideal of \(C^+(X)\) is strong \cite[Theorem~3.12]{BiswasFilomat}. By Lemma~\ref{lem:ordinary-radical-prime-intersection}, \(\sqrt{\langle g\rangle}\) is an intersection of prime ideals and is therefore strong. It is a strong radical ideal containing $g$, so
\([g]_{\mathfrak l}\subseteq\sqrt{\langle g\rangle}\). Conversely, \([g]_{\mathfrak l}\) is a radical ideal containing $g$, hence
\(\sqrt{\langle g\rangle}\subseteq[g]_{\mathfrak l}\). Thus
\([g]_{\mathfrak l}=\sqrt{\langle g\rangle}\). If \(f\in[g]_{\mathfrak l}\), then \(f^n=hg\) for some \(n\geq1\) and \(h\in C^+(X)\), which is impossible because
\(x^n e^{1/x^2}\) is unbounded as \(x\rightarrow0\).

If $X$ is a $P$-space, however, \(\Coz\) is a frame $\mathfrak l$-support. If
\(\Coz(f)\subseteq\Coz(g_1)\cup\cdots\cup\Coz(g_n)\), put
\(g=g_1+\cdots+g_n\). Then \(Z(g)\subseteq Z(f)\), and $Z(g)$ is clopen. The function
\[
h(x)=
\begin{cases}
f(x)/g(x),&g(x)>0,\\
0,&g(x)=0,
\end{cases}
\]
is continuous and satisfies \(f=hg\), proving the reflection condition.

For an arbitrary set $X$, the analogous map
\(\mathds{R}_+^X\to\mathcal P(X)\) is a frame $\mathfrak l$-support. For a measurable space \((X,\mathcal A)\), let \(\mathcal M^+(X,\mathcal A)\) denote the semiring of finite-valued nonnegative measurable functions. Its cozero map
\(\mathcal M^+(X,\mathcal A)\to\mathcal A\) is an $\mathfrak l$-support; it is a frame $\mathfrak l$-support when \(\mathcal A\) is closed under arbitrary unions. The same pointwise quotient proves the reflection condition, and indicator functions generate the codomain; see \cite{Measure}.
\end{enumerate}
\end{example}

Let \(\Id_{\mathfrak l}(S)\) denote the set of proper strong ideals and \(\Specl(S)\) the set of strong prime ideals, both with their hull--kernel topology.

\begin{proposition}\label{prop:strong-ideal-spectra}
The spaces \(\Id_{\mathfrak l}(S)\) and \(\Specl(S)\) are spectral.
\end{proposition}

\begin{proof}
Identify a subset \(A\subseteq S\) with its characteristic function in the Sierpiński cube \(\{0,1\}^S\), where \(\{0\}\) is open. This cube is spectral, and the coordinate conditions
\[
D(a)=\{A\mid a\notin A\},
\qquad
V(a)=\{A\mid a\in A\}
\]
are clopen in the patch topology. The proper ideals form the patch-closed set
\begin{align*}
\Id^{\mathrm{pr}}(S)
={}&D(1)\cap V(0)\\
&\cap\bigcap_{a,b\in S}\bigl(D(a)\cup V(ab)\bigr)\\
&\cap\bigcap_{a,b\in S}\bigl(D(a)\cup D(b)\cup V(a+b)\bigr).
\end{align*}
The strong-ideal condition is also patch closed:
\[
\Id_{\mathfrak l}(S)
 =\Id^{\mathrm{pr}}(S)\cap
  \bigcap_{a,b\in S}
  \bigl(D(a+b)\cup(V(a)\cap V(b))\bigr).
\]
The Sierpi\'nski cube is spectral, and every constructibly closed subspace of such a cube is spectral in the inherited topology \cite[Lemma~5.23.13]{StacksSpectral}. Here that inherited topology is exactly the hull--kernel topology. This proves the assertion for strong ideals.

Likewise,
\[
\Spec(S)
 =\Id^{\mathrm{pr}}(S)\cap
  \bigcap_{a,b\in S}
  \bigl(D(ab)\cup V(a)\cup V(b)\bigr)
\]
is patch closed. Hence
\(\Specl(S)=\Spec(S)\cap\Id_{\mathfrak l}(S)\) is patch closed and therefore spectral.
\end{proof}

For \(a\in S\), write
\[
D_{\mathfrak l,S}(a)=\{P\in\Specl(S)\mid a\notin P\}.
\]
These sets form a basis of quasi-compact open subsets of \(\Specl(S)\). Indeed, they are the intersections with \(\Specl(S)\) of the coordinate compact opens in the Sierpi\'nski cube used in Proposition~\ref{prop:strong-ideal-spectra}; equivalently, they are clopen in the constructible topology. Moreover,
\[
D_{\mathfrak l,S}(a)\cap D_{\mathfrak l,S}(b)
 =D_{\mathfrak l,S}(ab).
\]

\begin{proposition}\label{prop:strong-spectrum-functor}
Inverse image defines a contravariant functor
\[
\Specl\colon\CRig^{\op}\longrightarrow\SpecSp.
\]
For a morphism \(f\colon S\to T\), its action is
\[
f^*\colon\Specl(T)\longrightarrow\Specl(S),
\qquad P\longmapsto f^{-1}(P).
\]
\end{proposition}

\begin{proof}
The inverse image of a proper prime ideal is a proper prime ideal. If $P$ is strong and \(a+b\in f^{-1}(P)\), then
\(f(a)+f(b)=f(a+b)\in P\), so \(f(a),f(b)\in P\); hence \(a,b\in f^{-1}(P)\). Thus inverse image preserves strong prime ideals. Moreover,
\[
(f^*)^{-1}(D_{\mathfrak l,S}(a))
 =D_{\mathfrak l,T}(f(a)),
\]
so the induced map is spectral. Identity and composition follow from the corresponding laws for inverse image.
\end{proof}

\begin{remark}
The canonical $k$-support \(a\mapsto[a]_k\) need not preserve addition. By contrast, an $\mathfrak l$-support is precisely a semiring homomorphism from $S$ to the lattice semiring \((L,\vee,\wedge)\) satisfying the stated generation and reflection conditions.
\end{remark}

If \((L,\mathfrak D)\) is an $\mathfrak l$-support and $J$ is an ideal of $L$, then \(\mathfrak D^{-1}(J)\) is strong: from
\(\mathfrak D(a+b)=\mathfrak D(a)\vee\mathfrak D(b)\in J\), downward closure gives \(\mathfrak D(a),\mathfrak D(b)\in J\).

\begin{corollary}\label{cor:strong-support-spectrum}
For an $\mathfrak l$-support \((L,\mathfrak D)\), inverse image induces a homeomorphism
\[
\varphi^{\mathfrak l}_{\mathfrak D}\colon\Spec(L)\xrightarrow{\ \cong\ }\Specl(S),
\qquad
\mathfrak p\longmapsto\mathfrak D^{-1}(\mathfrak p).
\]
Its inverse sends a strong prime ideal \(\mathfrak q\) to the lattice ideal
\[
\bigl\langle\mathfrak D(a)\mid a\in\mathfrak q\bigr\rangle.
\]
\end{corollary}

\begin{proof}
For \(\mathfrak p\in\Spec(L)\), the preceding observation and the proof of Proposition~\ref{prop:support-spectrum-homeomorphism} show that \(\mathfrak D^{-1}(\mathfrak p)\) is a strong prime ideal. Conversely, let \(\mathfrak q\in\Specl(S)\), and put
\[
\mathfrak p_{\mathfrak q}=\bigl\langle\mathfrak D(a)\mid a\in\mathfrak q\bigr\rangle.
\]
The ideal \(\mathfrak q\) is radical because it is prime. If \(1\in\mathfrak p_{\mathfrak q}\), then
\(1\leq\bigvee_i\mathfrak D(q_i)\) for finitely many \(q_i\in\mathfrak q\), and the $\mathfrak l$-reflection axiom gives
\(1\in[q_1,\ldots,q_n]_{\mathfrak l}\subseteq\mathfrak q\), a contradiction. Thus \(\mathfrak p_{\mathfrak q}\) is proper.

Every element of $L$ is a finite join of support values. For
\(x=\bigvee_i\mathfrak D(a_i)\), the same argument as in \eqref{eq:support-membership-prime}, now using the $\mathfrak l$-reflection axiom, gives
\[
x\in\mathfrak p_{\mathfrak q}
\quad\Longleftrightarrow\quad
a_i\in\mathfrak q\text{ for every }i.
\]
To prove primality, suppose that \(x\wedge y\in\mathfrak p_{\mathfrak q}\) and \(x\notin\mathfrak p_{\mathfrak q}\). Write
\[
x=\bigvee_i\mathfrak D(a_i),\qquad
y=\bigvee_j\mathfrak D(b_j).
\]
The criterion supplies an index \(i_0\) with \(a_{i_0}\notin\mathfrak q\), while membership of \(x\wedge y\) gives \(a_i b_j\in\mathfrak q\) for all $i,j$. Since \(\mathfrak q\) is prime, every \(b_j\) lies in \(\mathfrak q\), and hence \(y\in\mathfrak p_{\mathfrak q}\). Thus \(\mathfrak p_{\mathfrak q}\) is prime. The same membership criterion gives
\[
\mathfrak D^{-1}(\mathfrak p_{\mathfrak q})=\mathfrak q.
\]
Conversely, if \(\mathfrak p\in\Spec(L)\), the lattice ideal generated by \(\mathfrak D[\mathfrak D^{-1}(\mathfrak p)]\) is visibly contained in \(\mathfrak p\); the reverse inclusion follows by writing an arbitrary \(x\in\mathfrak p\) as a finite join of support values, each of which is below $x$. Hence the two displayed assignments are inverse. Finally,
\[
(\varphi^{\mathfrak l}_{\mathfrak D})^{-1}(D_{\mathfrak l,S}(a))
 =D_L(\mathfrak D(a)),
\]
and these sets form bases. Hence the bijection is a homeomorphism.
\end{proof}

\subsection{Comparison of the three spectral functors}\label{subsec:three-spectra}

The ordinary, $k$-, and strong prime spectra form a nested functorial system. We now make this comparison precise at three levels: the specialization orders, the constructible topologies, and the associated frames, locales, and sheaf topoi. This resolves the comparison without imposing equality of the three notions of prime ideal.

\begin{proposition}\label{prop:nested-spectra}
For every \(S\in\CRig\), strong primality implies $k$-primality, and $k$-primality implies primality. Consequently there are natural spectral embeddings
\[
\lambda_S\colon\Specl(S)\hookrightarrow\Speck(S),
\qquad
\iota_S\colon\Speck(S)\hookrightarrow\Spec(S),
\]
which form natural transformations of contravariant functors
\[
\Specl\xRightarrow{\lambda}\Speck\xRightarrow{\iota}\Spec.
\]
For each of the three spectra, the specialization order is inclusion of ideals. Hence \(\lambda_S\), \(\iota_S\), and every inverse-image map induced by a semiring homomorphism are monotone; the two inclusions are order embeddings.
\end{proposition}

\begin{proof}
A strong ideal is a $k$-ideal, so the inclusions of underlying sets are immediate. The topologies on the smaller spectra are the subspace hull--kernel topologies, and
\[
D_{\mathfrak l,S}(a)=D_{k,S}(a)\cap\Specl(S),
\qquad
D_{k,S}(a)=D_S(a)\cap\Speck(S),
\]
so both inclusions are spectral embeddings. Inverse image preserves strong, $k$-, and ordinary prime ideals, and therefore the squares
\[
\begin{tikzcd}[column sep=4.5em]
\Specl(T) \arrow[r,hook,"\lambda_T"] \arrow[d,"f^*"']
& \Speck(T) \arrow[r,hook,"\iota_T"] \arrow[d,"f^*"']
& \Spec(T) \arrow[d,"f^*"]\\
\Specl(S) \arrow[r,hook,"\lambda_S"']
& \Speck(S) \arrow[r,hook,"\iota_S"']
& \Spec(S)
\end{tikzcd}
\]
commute.

For any one of these hull--kernel spectra, the closure of a point $P$ consists exactly of the prime ideals of the same type containing $P$. We use the convention that $P\leq_{\mathrm{sp}}Q$ means $Q\in\overline{\{P\}}$. With this convention,
\[
P\leq_{\mathrm{sp}}Q\quad\Longleftrightarrow\quad P\subseteq Q.
\]
The remaining order statements follow from preservation of inclusion by inverse image.
\end{proof}

For a spectral space $X$, write \(X^{\mathrm{con}}\) for its constructible, or patch, topology. It is generated by the quasi-compact opens and their complements and is a Stone topology \cite{Hochster,Johnstone}. For \(\tau\in\{\varnothing,k,\mathfrak l\}\), write \(D_{\tau,S}(a)\) for the corresponding basic open in the ordinary, $k$-, or strong prime spectrum, and write \(\mathcal V_{\tau,S}(a)\) for its complement.

\begin{proposition}\label{prop:constructible-spectra}
For every \(S\in\CRig\), the subsets \(\Speck(S)\) and \(\Specl(S)\) are closed in \(\Spec(S)^{\mathrm{con}}\), and \(\Specl(S)\) is closed in \(\Speck(S)^{\mathrm{con}}\). Their intrinsic constructible topologies agree with the corresponding subspace topologies. Consequently
\[
\Specl(S)^{\mathrm{con}}
 \xhookrightarrow{\ \lambda_S\ }
\Speck(S)^{\mathrm{con}}
 \xhookrightarrow{\ \iota_S\ }
\Spec(S)^{\mathrm{con}}
\]
is a natural chain of closed embeddings of Stone spaces. For every morphism \(f\colon S\to T\), each of the three maps $f^*$ is continuous for the constructible topologies.
\end{proposition}

\begin{proof}
Inside \(\Spec(S)\), the $k$-ideal and strong-ideal conditions can be written as
\[
\Speck(S)
 =\Spec(S)\cap
  \bigcap_{a,b\in S}
  \bigl(D(a+b)\cup D(a)\cup\mathcal V(b)\bigr),
\]
\[
\Specl(S)
 =\Spec(S)\cap
  \bigcap_{a,b\in S}
  \bigl(D(a+b)\cup(\mathcal V(a)\cap\mathcal V(b))\bigr).
\]
Every displayed factor is clopen in the constructible topology, so both subsets are constructibly closed. Since \(\Specl(S)\subseteq\Speck(S)\), it is also constructibly closed in \(\Speck(S)\).

The quasi-compact open bases on the three spectra are respectively
\(D_S(a)\), \(D_{k,S}(a)\), and \(D_{\mathfrak l,S}(a)\). Their complements are the corresponding hulls, so the constructible topology on each subspace is exactly the topology induced from \(\Spec(S)^{\mathrm{con}}\). Finally,
\[
(f^*)^{-1}(D_{\tau,S}(a))=D_{\tau,T}(f(a)),
\qquad
(f^*)^{-1}(\mathcal V_{\tau,S}(a))=\mathcal V_{\tau,T}(f(a))
\]
for \(\tau\in\{\varnothing,k,\mathfrak l\}\). Hence every inverse-image map is constructibly continuous, and naturality was proved in Proposition~\ref{prop:nested-spectra}.
\end{proof}

\begin{corollary}\label{cor:localic-sheaf-spectral-comparison}
Restriction of open sets gives natural transformations of coherent-frame-valued functors
\[
\Omega\circ\Spec
 \Longrightarrow
\Omega\circ\Speck
 \Longrightarrow
\Omega\circ\Specl,
\]
whose components are surjective coherent frame homomorphisms. Under the natural representations of Proposition~\ref{prop:natural-spectral-representations}, the first transformation is \(\mathfrak j\colon\RId\Rightarrow\RIdk\). Dually, write
\[
i_{\mathfrak l,k,S}\colon\Specl(S)\hookrightarrow\Speck(S),
\qquad
i_{k,S}\colon\Speck(S)\hookrightarrow\Spec(S).
\]
These inclusions determine natural locale embeddings and compatible geometric embeddings
\[
\Sh(\Specl(S))
 \xrightarrow{\ i_{\mathfrak l,k,S}\ }
\Sh(\Speck(S))
 \xrightarrow{\ i_{k,S}\ }
\Sh(\Spec(S)).
\]
Here an arrow labelled by $i$ denotes the geometric morphism
\(i^*\dashv i_*\), and its direct-image functor $i_*$ is fully faithful.
\end{corollary}

\begin{proof}
Inverse image along a spectral embedding preserves arbitrary unions, finite intersections, and quasi-compact opens. Because each map is a subspace inclusion, restriction of opens is surjective. The identification of the first transformation with $\mathfrak j$ is Proposition~\ref{prop:natural-k-radical-comparison}. Passing to opposite frame maps gives locale embeddings. For a locale embedding $i$, the induced geometric morphism \(i^*\dashv i_*\) has fully faithful direct image; applying this construction to the two natural inclusions gives the displayed geometric embeddings. See \cite{Johnstone,Che93,Jun17,Man19,Manuell}.
\end{proof}

\begin{corollary}\label{cor:strong-spectrum-density-criterion}
Put
\[
\mathfrak n_{\mathfrak l}(S)
 =\bigcap_{P\in\Specl(S)}P,
\]
with empty intersection interpreted as $S$. Then
\[
\overline{\Specl(S)}^{\,\Spec(S)}
 =\mathcal V_S(\mathfrak n_{\mathfrak l}(S)),
\qquad
\overline{\Specl(S)}^{\,\Speck(S)}
 =\mathcal V_{k,S}(\mathfrak n_{\mathfrak l}(S)).
\]
Under the standing reduced and conical hypotheses, \(\Specl(S)\) is dense in \(\Spec(S)\) if and only if it is dense in \(\Speck(S)\), and these conditions are equivalent to
\(\mathfrak n_{\mathfrak l}(S)=(0)\).
\end{corollary}

\begin{proof}
In a hull--kernel spectrum, the smallest closed set containing a family $Y$ of prime ideals is the hull of \(\bigcap_{P\in Y}P\). This gives the two closure formulas. Reducedness gives
\(\bigcap_{P\in\Spec(S)}P=(0)\), while the standing assumptions give
\(\bigcap_{P\in\Speck(S)}P=\sqrt[k]{(0)}=(0)\). The density equivalences follow.
\end{proof}

\subsection{\texorpdfstring{$r$}{r}-semirings and density of the \texorpdfstring{$k$}{k}-spectrum}\label{subsec:r-semirings}

This final subsection returns to ordinary commutative unital semirings and isolates a maximal-ideal criterion for spectral density. We first exhibit strict but dense inclusions of spectra, then prove the semisimple $r$-semiring criterion and record why its hypotheses are sufficient rather than necessary. The proposition proved here does not require the reduced or conical hypotheses imposed earlier.

The ordinary and $k$-prime spectra can differ even when the $k$-prime spectrum is dense. Consider
\[
S=\mathds{Q}_{\geq0}[\sqrt{2}]
 =\{a+b\sqrt{2}\mid a,b\in\mathds{Q}_{\geq0}\}
\]
with its usual operations. Its units are exactly the nonzero elements with one zero coefficient. Indeed, for $a,b>0$,
\[
a^{-1}=\frac1a,
\qquad
(b\sqrt{2})^{-1}=\frac{\sqrt{2}}{2b},
\]
whereas an equality
\((a+b\sqrt{2})(c+d\sqrt{2})=1\) with $a,b>0$ and $c,d\geq0$ would force the \(\sqrt{2}\)-coefficient
\(ad+bc\) to vanish; this is impossible because $c$ and $d$ cannot both be zero. Consequently, if a proper ideal contains a nonzero element with exactly one nonzero coefficient, it contains a unit. Thus a nonzero element of a proper ideal must have the form \(a+b\sqrt{2}\) with \(a,b>0\). Since \(a^2\neq2b^2\), exactly one of the following identities applies:
\[
(a+b\sqrt{2})\frac{a}{a^2-2b^2}
 =(a+b\sqrt{2})\frac{b\sqrt{2}}{a^2-2b^2}+1
 \quad\text{if }a^2>2b^2,
\]
\[
(a+b\sqrt{2})\frac{b\sqrt{2}}{2b^2-a^2}
 =(a+b\sqrt{2})\frac{a}{2b^2-a^2}+1
 \quad\text{if }2b^2>a^2.
\]
All coefficients in the applicable identity are nonnegative. If a proper $k$-ideal contained \(a+b\sqrt{2}\), the two products involving that element would lie in the ideal, and subtractivity would force $1$ into the ideal. Hence $(0)$ is the only proper $k$-ideal. Since $S$ has no zero divisors,
\[
\Speck(S)=\{(0)\}.
\]
On the other hand,
\[
\mathfrak m=\{0\}\cup
\{a+b\sqrt{2}\mid a,b\in\mathds{Q}_{>0}\}
\]
is a proper ideal whose complement is the multiplicative group of units. Thus \(\mathfrak m\) is prime and \(\Spec(S)\) has more than one point. The inclusion \(\Speck(S)\subsetneq\Spec(S)\) is nevertheless dense because $(0)$ is the generic point. This example also illustrates Corollary~\ref{cor:k-spectrum-dense}.

\begin{definition}[{\cite[Definition~3.3]{Chermnykh12}}]\label{def:r-semiring}
A semiring $S$ is an \emph{$r$-semiring} if every maximal ideal of $S$ is a $k$-ideal.
\end{definition}

Let \(\CRigr\) denote the full subcategory of \(\CRig\) consisting of $r$-semirings. The density statement below, however, remains valid for arbitrary commutative unital semirings and does not require the standing conical or reduced hypotheses.

\begin{proposition}\label{prop:r-semiring-density}
Let $S$ be a semisimple $r$-semiring, meaning
\[
\bigcap_{M\in\Max(S)}M=(0).
\]
Then the natural spectral embedding
\[
\Speck(S)\hookrightarrow\Spec(S)
\]
has dense image.
\end{proposition}

\begin{proof}
It suffices to show that every nonempty basic open \(D(a)\) meets \(\Speck(S)\). Nonemptiness implies \(a\neq0\). By semisimplicity there is a maximal ideal $M$ with \(a\notin M\). Since $S$ is an $r$-semiring, $M$ is a $k$-ideal. It is also prime: if \(xy\in M\) and \(x\notin M\), maximality gives an equality
\(1=m+sx\) with \(m\in M\), \(s\in S\), and hence
\[
y=my+sxy\in M.
\]
Thus \(M\in D(a)\cap\Speck(S)\).
\end{proof}

\begin{remark}
Under the standing reduced and conical hypotheses, density already follows from the natural nucleus comparison in Corollary~\ref{cor:k-spectrum-dense}. Proposition~\ref{prop:r-semiring-density} is retained because its direct maximal-ideal proof remains valid without those hypotheses and isolates the role of the $r$-condition.

The assumptions of Proposition~\ref{prop:r-semiring-density} are sufficient but not necessary. For \(S=\mathds{N}\), the prime ideals are
\[
(0),
\qquad
p\mathds{N}\quad(p\text{ prime}),
\qquad
\mathfrak m=\mathds{N}\setminus\{1\}.
\]
To see that the list is exhaustive, let $P$ be a nonzero prime ideal distinct from \(\mathfrak m\), and choose \(n>1\) outside $P$. Factoring a positive element of $P$ and applying primality repeatedly shows that $P$ contains a prime number $p$. It cannot contain a second prime \(q\neq p\), because \(p\mathds{N}+q\mathds{N}\) contains all sufficiently large integers and would therefore contain a sufficiently large power of $n$. Now \(p\mathds{N}\subseteq P\) because $p\in P$. Conversely, if $m\in P$ is positive, a prime factorization of $m$ and repeated primality produce a prime divisor $q$ of $m$ with $q\in P$. The uniqueness just proved forces $q=p$, so $p\mid m$. Hence \(P\subseteq p\mathds{N}\), and therefore \(P=p\mathds{N}\).

The $k$-prime ideals are $(0)$ and the ideals \(p\mathds{N}\), whereas \(\mathfrak m\) is not a $k$-ideal. Thus
\[
\Spec(\mathds{N})\setminus\Speck(\mathds{N})=\{\mathfrak m\}.
\]
Any open neighborhood of \(\mathfrak m\) contains a basic open \(D(n)\) with \(n\notin\mathfrak m\), hence \(n=1\); that basic open is the entire spectrum. Therefore \(\Speck(\mathds{N})\) is dense although \(\mathds{N}\) is neither semisimple nor an $r$-semiring.
\end{remark}

\subsection{A Stone-space criterion}

We specialize the support representation to positive cones of reduced $f$-rings. The aim is to identify, in a single equivalence theorem, the topological Stone condition, Booleanity of the support lattice, von Neumann regularity, zero-dimensionality of the radical frame, and maximality of prime ideals.

\begin{theorem}\label{thm:stone-spectrum-criterion}
Let $R$ be a commutative unital reduced $f$-ring with \(1\geq0\), put \(S=R^+\), and let \((L,\mathfrak D)\) be a support of $S$. The following are equivalent:
\begin{enumerate}
\item \(\Spec(S)\) is a Stone space;
\item $L$ is a Boolean algebra;
\item $S$ is von Neumann regular: for every \(a\in S\) there exists \(b\in S\) with \(a=a^2b\);
\item \(\RId(S)\) is zero-dimensional;
\item every prime ideal of $S$ is maximal.
\end{enumerate}
\end{theorem}

\begin{proof}
(1)$\Leftrightarrow$(2). Proposition~\ref{prop:support-spectrum-homeomorphism} identifies \(\Spec(S)\) with \(\Spec(L)\). The prime spectrum of a bounded distributive lattice is Hausdorff precisely when the lattice is Boolean \cite[Chapter~II, Proposition~4.4]{Johnstone}. A spectral Hausdorff space is a Stone space.

(2)$\Rightarrow$(3). By Lemma~\ref{lem:f-ring-idempotents} and Proposition~\ref{prop:support-idempotents}, every element of $L$ has the form \(\mathfrak D(e)\) for a unique idempotent \(e\in S\). Given $a$, choose $e$ with \(\mathfrak D(a)=\mathfrak D(e)\). The reflection axiom gives \([a]=[e]\). Since $e$ is idempotent and \(e\in[a]\), there exists \(b\in S\) with \(e=ab\). Since \(a\in[e]\), there are \(n\geq1\) and \(c\in S\) with \(a^n=ce\). Hence
\((ae^\perp)^n=0\), and reducedness yields \(ae^\perp=0\). Therefore
\[
a=a(e+e^\perp)=ae=a^2b.
\]

(3)$\Rightarrow$(4). If \(a=a^2b\) and \(e=ab\), then \(e^2=e\) and \([a]=[e]\). The radical ideals \([e]\) and \([e^\perp]\) are complementary. Since the principal radical ideals generate \(\RId(S)\), the frame is zero-dimensional.

(4)$\Rightarrow$(2). Proposition~\ref{prop:support-frame-reconstruction} identifies \(\RId(S)\) with \(\Id(L)\). The compact elements of \(\Id(L)\) are the principal ideals and identify with the elements of $L$. In a compact zero-dimensional frame, every compact element is a finite join of complemented elements and hence complemented. Its complement is compact: if \(c\vee d=1\), \(c\wedge d=0\), and \(d\leq\bigvee X\), compactness of $1$ applied to \(1\leq c\vee\bigvee X\) gives a finite subcover, which becomes a finite cover of $d$ after meeting with $d$. Thus complementation restricts to the compact elements, and $L$ is Boolean.

(1)$\Rightarrow$(5). In the Zariski spectrum, inclusion of prime ideals is the specialization order. Let $P$ be prime and choose a maximal ideal $M$ containing it. Such an $M$ is prime: if \(xy\in M\) and \(x\notin M\), maximality gives \(1=m+sx\) for some \(m\in M\) and \(s\in S\), whence \(y=my+sxy\in M\). If \(P\subsetneq M\), then $M$ is a proper specialization of $P$, contradicting Hausdorffness. Therefore \(P=M\), and every prime ideal is maximal.

(5)$\Rightarrow$(2). The homeomorphism of Proposition~\ref{prop:support-spectrum-homeomorphism} preserves the specialization order, so every prime ideal of $L$ is maximal. Hence $L$ is Boolean by \cite[Chapter~II, Corollary~4.9]{Johnstone}.
\end{proof}

\begin{corollary}
For every Tychonoff space $X$, the frame \(\RId(C^+(X))\) is zero-dimensional if and only if $X$ is a $P$-space.
\end{corollary}

\begin{proof}
This is Theorem~\ref{thm:stone-spectrum-criterion} together with \cite[Corollary~4.17]{BiswasFilomat}.
\end{proof}

\begin{corollary}
For every measurable space \((X,\mathcal A)\), let \(\mathcal M^+(X,\mathcal A)\) be the semiring of finite-valued nonnegative measurable functions. Then the frame
\(\RId(\mathcal M^+(X,\mathcal A))\) is zero-dimensional.
\end{corollary}

\begin{proof}
The ring of finite-valued real measurable functions is a commutative unital reduced $f$-ring with positive cone \(\mathcal M^+(X,\mathcal A)\). For \(f\geq0\), the measurable function
\[
g(x)=
\begin{cases}
1/f(x),&f(x)>0,\\
0,&f(x)=0
\end{cases}
\]
satisfies \(f=f^2g\). Thus the positive cone is von Neumann regular, and Theorem~\ref{thm:stone-spectrum-criterion} applies.
\end{proof}

\subsection{Universal support categories}

We next formulate the canonical representations as universal objects in three fixed-base categories. Finite and frame supports are treated by initiality, while open supports are classified by a terminal morphism to the prime spectrum.

Let \(\RId^{\mathrm{fin}}(S)\) be the bounded distributive lattice of finite radical ideals \([a_1,\ldots,a_n]\). Then
\((\RId^{\mathrm{fin}}(S),a\mapsto[a])\) is a support and
\((\RId(S),a\mapsto[a])\) is a frame support.

An \emph{open support} of $S$ is a pair \((X,d)\) consisting of a topological space and a map \(d\colon S\to\Omega(X)\) such that \(\{d(a)\mid a\in S\}\) is a basis and
\begin{enumerate}
\item \(d(0)=\varnothing\) and \(d(1)=X\);
\item \(d(ab)=d(a)\cap d(b)\);
\item \(d(a+b)\subseteq d(a)\cup d(b)\);
\item if \(d(a)\subseteq d(b_1)\cup\cdots\cup d(b_n)\), then
\(a\in[b_1,\ldots,b_n]\).
\end{enumerate}
A morphism \((X,d)\to(Y,e)\) is a continuous map \(u\colon X\to Y\) satisfying
\(u^{-1}(e(a))=d(a)\) for every $a$. These objects and morphisms form a category \(\OSupp(S)\).

\begin{proposition}\label{prop:universal-supports}
For every \(S\in\CRig\):
\begin{enumerate}
\item \((\RId^{\mathrm{fin}}(S),a\mapsto[a])\) is initial in \(\Supp(S)\);
\item \((\RId(S),a\mapsto[a])\) is initial in \(\FSupp(S)\);
\item \((\Spec(S),a\mapsto D(a))\) is terminal in \(\OSupp(S)\).
\end{enumerate}
\end{proposition}

\begin{proof}
(1) Let \((L,\mathfrak D)\) be a support. Define
\[
\widetilde{\mathfrak D}([a_1,\ldots,a_n])
 =\mathfrak D(a_1)\vee\cdots\vee\mathfrak D(a_n).
\]
This is well defined. If
\([a_1,\ldots,a_m]=[b_1,\ldots,b_n]\), then each $a_i$ belongs to the radical ideal on the right, so Proposition~\ref{prop:support-properties}(5) gives
\(\mathfrak D(a_i)\leq\bigvee_j\mathfrak D(b_j)\); the reverse inequality is symmetric. Furthermore,
\[
[a_1,\ldots,a_m]\vee[b_1,\ldots,b_n]
 =[a_1,\ldots,a_m,b_1,\ldots,b_n]
\]
and
\[
[a_1,\ldots,a_m]\cap[b_1,\ldots,b_n]
 =\sqrt{\left\langle
 a_i b_j:1\leq i\leq m,\ 1\leq j\leq n
 \right\rangle}.
\]
The support axioms and distributivity in $L$ show that
\(\widetilde{\mathfrak D}\) is a bounded lattice homomorphism. It is uniquely determined by
\(\widetilde{\mathfrak D}([a])=\mathfrak D(a)\), proving initiality.

(2) Let \((F,\mathfrak D)\) be a frame support and define
\[
\widehat{\mathfrak D}(I)=\bigvee_{a\in I}\mathfrak D(a).
\]
Proposition~\ref{prop:support-properties}(5) gives
\(\widehat{\mathfrak D}([a])=\mathfrak D(a)\). Let \((I_\lambda)\) be a family of radical ideals. Since
\(I_\lambda\subseteq\bigvee_\lambda I_\lambda\), monotonicity gives
\[
\bigvee_\lambda\widehat{\mathfrak D}(I_\lambda)
 \leq
\widehat{\mathfrak D}\left(\bigvee_\lambda I_\lambda\right).
\]
Conversely, if
\(a\in\sqrt{\sum_\lambda I_\lambda}\), then for some $n\geq1$ one can write
\(a^n=y_1+\cdots+y_r\) with each $y_j$ belonging to one of the $I_\lambda$. Hence
\[
\mathfrak D(a)=\mathfrak D(a^n)
 \leq\bigvee_{j=1}^r\mathfrak D(y_j)
 \leq\bigvee_\lambda\widehat{\mathfrak D}(I_\lambda).
\]
Taking the join over all such $a$ proves the reverse inequality and therefore
\[
\widehat{\mathfrak D}\left(\bigvee_\lambda I_\lambda\right)
 =\bigvee_\lambda\widehat{\mathfrak D}(I_\lambda).
\]
Moreover,
\begin{align*}
\widehat{\mathfrak D}(I)\wedge\widehat{\mathfrak D}(J)
 &=\bigvee_{a\in I,\,b\in J}\mathfrak D(ab)\\
 &\leq\widehat{\mathfrak D}(I\cap J),
\end{align*}
while the reverse inequality follows because every \(c\in I\cap J\) contributes a support value below both joins. Thus \(\widehat{\mathfrak D}\) is a frame homomorphism. Since
\(I=\bigvee_{a\in I}[a]\), it is unique.

(3) Let \((X,d)\) be an open support. For \(x\in X\), put
\[
u(x)=\{a\in S\mid x\notin d(a)\}.
\]
The open-support axioms show this directly. One has \(0\in u(x)\) and \(1\notin u(x)\). If \(a,b\in u(x)\), then
\(x\notin d(a)\cup d(b)\), and the inclusion
\(d(a+b)\subseteq d(a)\cup d(b)\) gives \(a+b\in u(x)\). If \(a\in u(x)\) and \(s\in S\), then
\(d(sa)=d(s)\cap d(a)\), so \(sa\in u(x)\). Finally, if \(ab\in u(x)\), then
\(x\notin d(a)\cap d(b)\), whence \(a\in u(x)\) or \(b\in u(x)\). Thus \(u(x)\) is a proper prime ideal. Moreover,
\[
u^{-1}(D(a))=d(a)
\]
for every $a$. Hence $u$ is continuous and is a morphism of open supports. Any such morphism is determined by this membership condition, so $u$ is unique.
\end{proof}

The proposition packages the support theory into three universal statements: finite lattice support and frame support are free representations, while the prime spectrum is a classifying space for open supports.

\subsection{The fibred category of finite supports}\label{subsec:support-bifibration}

We now allow the base semiring to vary. The reflection axiom is strong enough to produce a unique map of support lattices over every semiring homomorphism; consequently the fixed-base categories \(\Supp(S)\) are the fibers of a particularly rigid Grothendieck bifibration. The canonical support functor supplies its distinguished cleavage.

\begin{definition}\label{def:total-support-category}
Let \(\TotSupp\) be the category whose objects are triples
\((S,L,\mathfrak D)\), where \(S\in\CRig\) and \((L,\mathfrak D)\in\Supp(S)\). A morphism
\[
(f,h)\colon(S,L,\mathfrak D)\longrightarrow(T,M,\mathfrak E)
\]
consists of a semiring homomorphism \(f\colon S\to T\) and a bounded lattice homomorphism \(h\colon L\to M\) satisfying
\[
h\mathfrak D=\mathfrak E f.
\]
Composition is componentwise. The projection
\[
\pi\colon\TotSupp\longrightarrow\CRig,
\qquad
(S,L,\mathfrak D)\longmapsto S,
\]
is called the \emph{total support projection}.
\end{definition}

\begin{lemma}\label{lem:support-change-of-scalars}
Let \(f\colon S\to T\) be a morphism in \(\CRig\), let
\((L,\mathfrak D)\in\Supp(S)\), and let
\((M,\mathfrak E)\in\Supp(T)\). There is a unique bounded lattice homomorphism
\[
h_f\colon L\longrightarrow M
\]
such that \(h_f\mathfrak D=\mathfrak E f\). If
\(x=\bigvee_{i=1}^n\mathfrak D(a_i)\), then
\[
h_f(x)=\bigvee_{i=1}^n\mathfrak E(f(a_i)).
\]
\end{lemma}

\begin{proof}
Every element of $L$ is a finite join of support values: a bounded distributive lattice term in the generators can be put into disjunctive normal form, and finite meets satisfy
\[
\mathfrak D(a_1)\wedge\cdots\wedge\mathfrak D(a_r)
 =\mathfrak D(a_1\cdots a_r).
\]
Suppose
\(\bigvee_i\mathfrak D(a_i)\leq\bigvee_j\mathfrak D(b_j)\). The reflection axiom gives
\(a_i\in[b_1,\ldots,b_m]\) for every $i$. Applying $f$ yields
\(f(a_i)\in[f(b_1),\ldots,f(b_m)]\), and Proposition~\ref{prop:support-properties}(5), applied to $\mathfrak E$, gives
\[
\mathfrak E(f(a_i))
 \leq\bigvee_j\mathfrak E(f(b_j)).
\]
Thus the displayed formula is independent of the chosen representation of $x$.

It plainly preserves finite joins and the bounds. If
\(x=\bigvee_i\mathfrak D(a_i)\) and
\(y=\bigvee_j\mathfrak D(b_j)\), then distributivity and multiplicativity give
\begin{align*}
h_f(x\wedge y)
 &=\bigvee_{i,j}\mathfrak E(f(a_i b_j))\\
 &=\left(\bigvee_i\mathfrak E(f(a_i))\right)
   \wedge
   \left(\bigvee_j\mathfrak E(f(b_j))\right)
 =h_f(x)\wedge h_f(y).
\end{align*}
Hence $h_f$ is a bounded lattice homomorphism and satisfies the required equation. Uniqueness follows because the values of $\mathfrak D$ generate $L$.
\end{proof}

\begin{theorem}\label{thm:support-bifibration}
The total support projection
\[
\pi\colon\TotSupp\longrightarrow\CRig
\]
is full, faithful, and essentially surjective. It is both a Grothendieck fibration and a Grothendieck opfibration, and every morphism of \(\TotSupp\) is both cartesian and cocartesian. Consequently $\pi$ is an equivalence of categories, and its fiber over $S$ is \(\Supp(S)\). The canonical compact-open support defines a functor
\[
\mathsf s\colon\CRig\longrightarrow\TotSupp,
\qquad
S\longmapsto\bigl(S,\KOpen(\Spec(S)),\mathfrak D_S\bigr),
\]
which is a bicartesian section of $\pi$ and a quasi-inverse to it.

For every \(f\colon S\to T\), cartesian and cocartesian change of scalars between the fibers exist and are unique up to unique natural isomorphism. On the canonical compact-open supports, the structural arrow of $\mathsf s(f)$ is
\[
\mathfrak D(f)\colon\KOpen(\Spec(S))\longrightarrow\KOpen(\Spec(T)),
\qquad
D_S(a)\longmapsto D_T(f(a)).
\]
Equivalently, under finite radical ideals it is
\([a_1,\ldots,a_n]\mapsto[f(a_1),\ldots,f(a_n)]\).
\end{theorem}

\begin{proof}
Lemma~\ref{lem:support-change-of-scalars} gives, over every semiring homomorphism $f$ and between every chosen pair of supports over its source and target, exactly one morphism of \(\TotSupp\). Therefore
\[
\operatorname{Hom}_{\TotSupp}
 ((S,L,\mathfrak D),(T,M,\mathfrak E))
 \xrightarrow{\ \pi\ }
\operatorname{Hom}_{\CRig}(S,T)
\]
is a bijection. Thus $\pi$ is full and faithful. It is essentially surjective because every $S$ has its canonical compact-open support.

Let $u\colon A\to B$ lie over $f\colon S\to T$. If
$v\colon C\to B$ lies over $f g$, Lemma~\ref{lem:support-change-of-scalars} supplies a unique arrow $w\colon C\to A$ over $g$. The composites $u w$ and $v$ lie over the same base morphism and are therefore equal by uniqueness. Hence $u$ is cartesian. The dual factorization argument shows that $u$ is cocartesian. Thus every arrow is bicartesian and $\pi$ is a bifibration in the sense of Grothendieck \cite{GrothendieckFib}.

The fiber over $S$ consists exactly of supports of $S$ and their fixed-base morphisms, so it is \(\Supp(S)\). Chosen cartesian and cocartesian lifts define the two change-of-scalars pseudofunctors. Their uniqueness up to unique natural isomorphism follows from the fact that every fiber is a contractible groupoid.

The canonical compact-open support is functorial by Remark~\ref{rem:canonical-support-functor}; hence it defines a section $\mathsf s$ of $\pi$. Every arrow in $\TotSupp$ is bicartesian, so this section is bicartesian. For an object $A=(S,L,\mathfrak D)$, Lemma~\ref{lem:support-change-of-scalars} gives a unique arrow
\[
\epsilon_A\colon\mathsf s\pi(A)\longrightarrow A
\]
over $1_S$. It is an isomorphism, since the unique arrow in the opposite direction has composites lying over $1_S$ and uniqueness forces both composites to be identities. These arrows are natural for the same uniqueness reason. Thus $\mathsf s$ is a quasi-inverse to $\pi$. The displayed formula for $\mathsf s(f)$ is Remark~\ref{rem:canonical-support-functor}; the finite-radical formula follows from Theorem~\ref{thm:radical-functor} restricted to compact elements.
\end{proof}

\begin{remark}
The categorical content of Theorem~\ref{thm:support-bifibration} lies in the base-change formula and the canonical bicartesian section. The equivalence itself reflects the rigidity imposed by finite generation and the reflection axiom: each fiber \(\Supp(S)\) is contractible. Frame and open supports retain their initial or terminal objects from Proposition~\ref{prop:universal-supports}; however, the proof does not extend automatically to them, because finite reflection does not control arbitrary joins or arbitrary base change. Additional hypotheses would be required before asserting an analogous bifibration statement.
\end{remark}

\section{Quantale completion and monadicity}\label{sec:quantale-completion}

This section isolates the ideal-completion part of the theory. We first construct the $k$-ideal quantale functor on complete idealic semirings and recall its free universal property. We then compute the induced monad, compare its frame restriction with the classical ideal-lattice monad, and identify its Eilenberg--Moore category. Throughout this section, semirings are commutative, unital, complete, idealic, and additively idempotent; reducedness and conicality are not assumed.

\subsection{Complete idealic semirings and the \texorpdfstring{$k$}{k}-ideal functor}

We first identify the idempotent-semiring setting in which all $k$-ideals form a quantale. After fixing the natural order and principal $k$-ideals, we define extension on morphisms and prove that the construction is a functor to integral commutative quantales.

An additively idempotent semiring \(\ddot S\) carries its natural order
\[
x\leq y\quad\Longleftrightarrow\quad x+y=y.
\]
We use a dotted symbol to distinguish such semirings from the reduced conical semirings of the preceding sections.

\begin{definition}\label{def:complete-idealic-semiring}
A commutative unital additively idempotent semiring \(\ddot S\) is
\begin{enumerate}
\item \emph{complete} if every subset has a supremum in the natural order;
\item \emph{idealic} if \(x\leq1\) for every \(x\in\ddot S\).
\end{enumerate}
\end{definition}

Complete idealic semirings appear in the spectral theory of idempotent semirings and in constructions related to geometry over \(\mathds{F}_1\); see \cite{Takagi,Ray22}. We write \(\CIdRig\) for the category whose objects are commutative unital complete idealic semirings and whose morphisms are unital semiring homomorphisms. The morphisms are not assumed to preserve arbitrary suprema.

\begin{lemma}\label{lem:RIdk-complete-idealic}
For every \(S\in\CRig\), the frame \(\RIdk(S)\), with addition given by \(\bigvee_k\) and multiplication given by intersection, is a complete idealic semiring.
\end{lemma}

\begin{proof}
Theorem~\ref{thm:k-radical-coherent} shows that \(\RIdk(S)\) is a frame. Regard its join as addition and its meet as multiplication. The bottom is the additive identity, the top $S$ is the multiplicative identity, and frame distributivity gives the semiring distributive laws. Arbitrary joins exist and every element lies below the multiplicative unit.
\end{proof}

Let $A$ be a complete idealic semiring. A proper element \(p<1\) is \emph{prime} if
\[
xy\leq p\quad\Longrightarrow\quad x\leq p\text{ or }y\leq p.
\]
The \emph{Zariski space} \(\Zar(A)\) is the set of prime elements, with basic opens
\[
\mathbb D_A(x)=\{p\in\Zar(A)\mid x\nleq p\}.
\]

\begin{proposition}[{\cite[Proposition~3.40]{Ray22}}]\label{prop:zar-ridk}
For every \(S\in\CRig\), identifying the multiplication on \(\RIdk(S)\) with intersection, the assignment
\[
P\longmapsto P
\]
induces a homeomorphism
\[
\Speck(S)\xrightarrow{\ \cong\ }
\Zar\bigl(\RIdk(S),\bigvee_k,\cap\bigr).
\]
\end{proposition}

\begin{proof}
Every $k$-prime ideal is $k$-radical and is a prime element of \(\RIdk(S)\): if \(I\cap J\subseteq P\) and neither $I$ nor $J$ is contained in $P$, then choosing \(x\in I\setminus P\) and \(y\in J\setminus P\) gives \(xy\in I\cap J\setminus P\), a contradiction. Conversely, let \(P\in\RIdk(S)\) be a prime element. If \(ab\in P\), then
\[
[a]_k\cap[b]_k=[ab]_k\subseteq P,
\]
so \([a]_k\subseteq P\) or \([b]_k\subseteq P\), and therefore \(a\in P\) or \(b\in P\). Thus $P$ is $k$-prime. The bijection is a homeomorphism because
\[
\mathbb D_{\RIdk(S)}([a]_k)
 =\{P\in\Speck(S)\mid a\notin P\}
 =D_k(a).
\]
For an arbitrary \(I\in\RIdk(S)\), one has
\[
\mathbb D_{\RIdk(S)}(I)=\bigcup_{a\in I}\mathbb D_{\RIdk(S)}([a]_k),
\]
so these principal basic opens form a basis. The displayed bijection is therefore a homeomorphism.
\end{proof}

For an additively idempotent semiring, an ideal is a $k$-ideal exactly when it is downward closed in the natural order \cite[Proposition~3.11]{Ray22}. Hence
\[
\langle x\rangle_k
 =\{y\in\ddot S\mid y\leq xr\text{ for some }r\in\ddot S\}.
\]
For $k$-ideals \(I,J\), write
\[
I\vee^k J=\mathcal C_k(I+J),
\]
and define arbitrary joins analogously. Their multiplication is the $k$-ideal product
\[
I\odot_kJ=\mathcal C_k(I\cdot J)
\]
introduced in Section~\ref{sec:preliminaries}.

\begin{lemma}\label{lem:principal-k-ideals}
For all \(x,y\in\ddot S\),
\[
\langle x+y\rangle_k
 =\langle x\rangle_k\vee^k\langle y\rangle_k,
\qquad
\langle xy\rangle_k
 =\langle x\rangle_k\odot_k\langle y\rangle_k.
\]
\end{lemma}

\begin{proof}
Since \(x,y\leq x+y\), both principal $k$-ideals on the right of the first equality lie in \(\langle x+y\rangle_k\). Conversely, \(x+y\) belongs to
\(\langle x\rangle_k+\langle y\rangle_k\), so minimality of principal $k$-ideals gives the reverse inclusion.

The $k$-ideal product \(\langle x\rangle_k\odot_k\langle y\rangle_k\) is a $k$-ideal containing $xy$, and therefore it contains \(\langle xy\rangle_k\). Conversely, if \(u\leq xr\) and \(v\leq ys\), then
\(uv\leq xy(rs)\), so every generator $uv$ of the ordinary ideal product belongs to \(\langle xy\rangle_k\). Since \(\langle xy\rangle_k\) is a $k$-ideal, it contains the $k$-closure of that ordinary product. This proves the reverse inclusion.
\end{proof}

With \(\odot_k\) as multiplication and $k$-ideal join as addition, \(\Idk(\ddot S)\) is a commutative unital quantale \cite[Lemmas~3.27 and~3.28]{Ray22}. Its unit is \(\langle1\rangle_k=\ddot S\), because idealicity gives \(x\leq1\) for every $x$.

\begin{definition}\label{def:unitally-bounded-quantale}
A commutative unital quantale \((Q,\vee,*,1)\) is \emph{integral} if its unit $1$ is the maximum element. Such quantales are also called unitally bounded.
\end{definition}

Let \(\QuantTop\) be the category of integral commutative quantales and homomorphisms preserving arbitrary joins, multiplication, and the unit. Every object of \(\QuantTop\) has an underlying complete idealic semiring, so there is a forgetful functor
\[
\mathcal U\colon\QuantTop\longrightarrow\CIdRig.
\]
The converse implication on objects is false: completeness and idealicity do not force arbitrary distributivity.

\begin{example}\label{ex:complete-idealic-not-quantale}
Let
\[
C=\Omega(\mathds{R})^{\op},
\qquad
U+V=U\cap V,
\qquad
U\cdot V=U\cup V.
\]
Then \(0_C=\mathds{R}\), \(1_C=\varnothing\), and joins in the natural order are interiors of intersections. Thus $C$ is a complete idealic semiring. Put
\[
U=\mathds{R}\setminus\{0\},
\qquad
V_n=\left(-\frac1n,\frac1n\right).
\]
Since
\(\bigvee^C_{n\geq1}V_n=\operatorname{int}(\bigcap_{n\geq1}V_n)=\varnothing\),
\[
U\cdot\left(\bigvee^C_{n\geq1}V_n\right)=U,
\qquad
\bigvee^C_{n\geq1}(U\cdot V_n)=\mathds{R}.
\]
Therefore multiplication fails to distribute over this arbitrary join, and $C$ is not a quantale.
\end{example}

For a morphism \(h\colon\ddot S\to\ddot T\) in \(\CIdRig\), define
\[
\Idk(h)(I)
 =\mathcal C_k(\langle h[I]\rangle)
 =\mathop{\bigvee{}^k}_{x\in I}\langle h(x)\rangle_k.
\]
This is again best viewed as extension left adjoint to contraction on $k$-ideals.

\begin{proposition}\label{prop:Idk-functor}
The assignments
\[
\ddot S\longmapsto\Idk(\ddot S),
\qquad
h\longmapsto\Idk(h)
\]
define a functor
\[
\Idk\colon\CIdRig\longrightarrow\QuantTop.
\]
For every \(x\in\ddot S\),
\[
\Idk(h)(\langle x\rangle_k)=\langle h(x)\rangle_k.
\]
\end{proposition}

\begin{proof}
Put
\(E=\Idk(h)(\langle x\rangle_k)\). Since $h(x)$ is among the generators of $E$, minimality of the principal $k$-ideal gives
\(\langle h(x)\rangle_k\subseteq E\). Conversely, if
\(y\in\langle x\rangle_k\), then \(y\leq xr\) for some $r\in\ddot S$, and therefore
\(h(y)\leq h(x)h(r)\). Thus every element of
\(h[\langle x\rangle_k]\) belongs to \(\langle h(x)\rangle_k\). Since the latter is a $k$-ideal, it contains the ideal generated by this image and its $k$-closure. Hence
\(E\subseteq\langle h(x)\rangle_k\), proving the principal formula.

Every $k$-ideal is the join of its principal $k$-ideals. Let
\(J=\bigvee^k_\lambda I_\lambda\). If \(z\in J\), then there are finitely many
\(y_i\in\bigcup_\lambda I_\lambda\) and coefficients \(r_i\in\ddot S\) such that
\[
z\leq\sum_i r_i y_i\leq\sum_i y_i;
\]
the second inequality uses \(r_i\leq1\). Applying $h$ shows that
\(\langle h(z)\rangle_k\) lies below
\(\bigvee^k_\lambda\Idk(h)(I_\lambda)\). The reverse inequality is immediate, so \(\Idk(h)\) preserves arbitrary joins.

Moreover,
\[
I\odot_kJ=\mathop{\bigvee{}^k}_{x\in I,\,y\in J}\langle xy\rangle_k,
\]
because both sides are the least $k$-ideal containing all pairwise products. Together with Lemma~\ref{lem:principal-k-ideals}, this identity proves preservation of multiplication. The unit is preserved because
\[
\langle1\rangle_k=\ddot S,
\qquad
\Idk(h)(\langle1\rangle_k)=\langle h(1)\rangle_k=\ddot T.
\]
Hence \(\Idk(h)\) is a quantale homomorphism. Since join-preserving maps out of \(\Idk(\ddot S)\) are determined on principal $k$-ideals, the identity and composition laws follow from the principal formula.
\end{proof}

\subsection{The free integral-quantale adjunction}

Ideal completion is a standard free-quantale construction for idempotent or ordered semirings; see \cite{NishizawaFurusawa14,Fujii23}. We record the exact $k$-ideal version needed here, because its unit, counit, and action on morphisms determine the monad studied in the next two subsections. The proof constructs the adjoint extension explicitly and verifies arbitrary-join and multiplicative preservation.

The map
\[
\eta_{\ddot S}\colon\ddot S\longrightarrow\mathcal U\Idk(\ddot S),
\qquad
x\longmapsto\langle x\rangle_k,
\]
is a unital semiring homomorphism by Lemma~\ref{lem:principal-k-ideals}. It is the unit of the following adjunction.

\begin{theorem}\label{thm:k-ideal-quantale-adjunction}
The functor \(\Idk\) is left adjoint to \(\mathcal U\):
\[
\Idk\dashv\mathcal U.
\]
Equivalently, for \(\ddot S\in\CIdRig\) and \(Q\in\QuantTop\), there is a natural bijection
\[
\operatorname{Hom}_{\QuantTop}(\Idk(\ddot S),Q)
\cong
\operatorname{Hom}_{\CIdRig}(\ddot S,\mathcal UQ).
\]
\end{theorem}

\begin{proof}
Let \(f\colon\ddot S\to\mathcal UQ\) be a unital semiring homomorphism. Define
\[
\Phi_f\colon\Idk(\ddot S)\longrightarrow Q,
\qquad
\Phi_f(I)=\bigvee_{x\in I}f(x).
\]
If \(y\in\langle x\rangle_k\), then \(y\leq xr\) for some \(r\in\ddot S\). Since $f$ is monotone and \(f(r)\leq1\),
\[
f(y)\leq f(x)*f(r)\leq f(x).
\]
Because \(x\in\langle x\rangle_k\),
\begin{equation}\label{eq:Phi-principal}
\Phi_f(\langle x\rangle_k)=f(x).
\end{equation}
In particular, \(\Phi_f(\ddot S)=1\).

Let \((I_\lambda)\) be a family of $k$-ideals and put
\(J=\bigvee^k_\lambda I_\lambda\). If \(z\in J\), choose a finite inequality
\[
z\leq\sum_i r_i y_i\leq\sum_i y_i,
\qquad
y_i\in\bigcup_\lambda I_\lambda.
\]
Then
\[
f(z)\leq\bigvee_i f(y_i)
 \leq\bigvee_\lambda\Phi_f(I_\lambda).
\]
The reverse inequality follows from \(I_\lambda\subseteq J\), so \(\Phi_f\) preserves arbitrary joins.

Let \(z\in I\odot_kJ\). Since $k$-closure is downward closure in the natural order, there are \(x_1,\ldots,x_n\in I\), \(y_1,\ldots,y_n\in J\), and \(r_1,\ldots,r_n\in\ddot S\) such that
\[
z\leq\sum_{i=1}^n r_i x_i y_i
 \leq\sum_{i=1}^n x_i y_i;
\]
the second inequality uses \(r_i\leq1\). Hence
\[
f(z)\leq\bigvee_{i=1}^n f(x_i)*f(y_i)
 \leq\Phi_f(I)*\Phi_f(J).
\]
Taking the join over all $z$ gives
\[
\Phi_f(I\odot_kJ)\leq\Phi_f(I)*\Phi_f(J).
\]
Conversely, every \(xy\) with \(x\in I\) and \(y\in J\) belongs to \(I\odot_kJ\), and quantale distributivity gives
\[
\Phi_f(I)*\Phi_f(J)
 =\bigvee_{x\in I,\,y\in J}f(xy)
 \leq\Phi_f(I\odot_kJ).
\]
Also, \(\langle0\rangle_k=(0)\) and \(\langle1\rangle_k=\ddot S\), so \(\Phi_f\) preserves bottom and the multiplicative unit. Thus \(\Phi_f\) is a quantale homomorphism, and \eqref{eq:Phi-principal} says
\(\Phi_f\eta_{\ddot S}=f\).

This extension is unique because
\[
I=\mathop{\bigvee{}^k}_{x\in I}\langle x\rangle_k.
\]
Conversely, for \(\Psi\colon\Idk(\ddot S)\to Q\) in \(\QuantTop\), define
\[
f_\Psi(x)=\Psi(\langle x\rangle_k).
\]
Lemma~\ref{lem:principal-k-ideals}, together with
\[
\langle0\rangle_k=(0),\qquad \langle1\rangle_k=\ddot S,
\]
and preservation by \(\Psi\) of arbitrary joins, multiplication, bottom, and unit, shows that \(f_\Psi\) is a unital semiring homomorphism. Equation~\eqref{eq:Phi-principal} gives \(f_{\Phi_f}=f\), while
\[
\Phi_{f_\Psi}(I)
 =\bigvee_{x\in I}\Psi(\langle x\rangle_k)
 =\Psi\left(\mathop{\bigvee{}^k}_{x\in I}\langle x\rangle_k\right)
 =\Psi(I).
\]
The constructions are inverse. To record naturality explicitly, let \(u\colon\ddot S'\to\ddot S\) be a morphism in \(\CIdRig\) and let \(v\colon Q\to Q'\) be a morphism in \(\QuantTop\). For every $k$-ideal $I$,
\begin{align*}
(\Phi_f\circ\Idk(u))(I)
 &=\Phi_f\left(\mathop{\bigvee{}^k}_{x\in I}\langle u(x)\rangle_k\right)\\
 &=\bigvee_{x\in I} f(u(x))
  =\Phi_{f\circ u}(I),\\
(v\circ\Phi_f)(I)
 &=v\left(\bigvee_{x\in I}f(x)\right)
  =\bigvee_{x\in I}v(f(x))
  =\Phi_{v\circ f}(I).
\end{align*}
Thus the hom-set bijection is natural in both variables.
\end{proof}

\begin{corollary}\label{cor:unit-counit}
The unit of the adjunction is
\[
\eta_{\ddot S}(x)=\langle x\rangle_k,
\]
and the counit at \(Q\in\QuantTop\) is
\[
\varepsilon_Q\colon\Idk(\mathcal UQ)\longrightarrow Q,
\qquad
\varepsilon_Q(I)=\bigvee_{x\in I}x.
\]
\end{corollary}

\begin{proof}
The unit was identified above. The counit is the extension \(\Phi_{1_{\mathcal UQ}}\) corresponding to the identity morphism under the hom-set bijection of Theorem~\ref{thm:k-ideal-quantale-adjunction}.
\end{proof}

\subsection{The induced monad and comparison with ideal completion}\label{subsec:induced-monad}

The adjunction determines a monad \(\mathbb T=\mathcal U\Idk\) on \(\CIdRig\). We compute its unit and multiplication and then identify its precise overlap with the classical ideal-lattice monad. This comparison separates the genuinely multiplicative complete-idealic setting from the frame case in which multiplication is meet.

For \(\ddot S\in\CIdRig\), idealicity gives
\[
\langle x\rangle_k=\mathord\downarrow x
 =\{y\in\ddot S\mid y\leq x\}.
\]
For a subset \(A\subseteq\ddot S\), put
\[
K_{\ddot S}(A)
 =\mathop{\bigvee{}^k}_{a\in A}\mathord\downarrow a.
\]

\begin{lemma}\label{lem:k-generated-supremum}
For every \(A\subseteq\ddot S\),
\[
\bigvee_{x\in K_{\ddot S}(A)}x=\bigvee_{a\in A}a.
\]
Moreover, for every \(x\in\ddot S\),
\[
(\mathord\downarrow x)\odot_k K_{\ddot S}(A)
 =K_{\ddot S}(\{xa\mid a\in A\}).
\]
\end{lemma}

\begin{proof}
The inclusion \(A\subseteq K_{\ddot S}(A)\) gives one inequality. Conversely, if \(z\in K_{\ddot S}(A)\), then for finitely many \(a_i\in A\) and \(r_i\in\ddot S\),
\[
z\leq\sum_i r_i a_i\leq\sum_i a_i\leq\bigvee_{a\in A}a,
\]
because \(r_i\leq1\). This proves the first formula.

The product on the left is a $k$-ideal containing every $xa$, so it contains the right-hand side. Conversely, let \(u\leq x\) and \(v\in K_{\ddot S}(A)\). Choose finitely many \(a_i\in A\) with
\(v\leq\sum_i a_i\). Finite distributivity gives
\[
uv\leq xv\leq\sum_i xa_i,
\]
so every generator of the ordinary product lies in
\(K_{\ddot S}(\{xa\mid a\in A\})\). Since the latter is a $k$-ideal, it contains the $k$-closure of that product.
\end{proof}

The monad has
\[
\mathbb T(\ddot S)=\Idk(\ddot S),
\qquad
\eta_{\ddot S}(x)=\mathord\downarrow x,
\]
and, for a $k$-ideal \(\mathcal J\) of the complete idealic semiring \(\Idk(\ddot S)\),
\[
\mu_{\ddot S}(\mathcal J)
 =\mathop{\bigvee{}^k}_{I\in\mathcal J}I.
\]
The last formula is the underlying map of the counit at \(\Idk(\ddot S)\).

Let \(\DLatFrm\) be the full subcategory of \(\DLat\) whose objects are frames; its morphisms are bounded lattice homomorphisms and are not required to preserve arbitrary joins. There is a fully faithful functor
\[
\mathcal E\colon\DLatFrm\longrightarrow\CIdRig,
\qquad
F\longmapsto(F,\vee,\wedge,0,1),
\]
which regards meet as multiplication. Let \(\mathbb I\) denote the ideal-lattice monad on \(\DLat\).

\begin{proposition}\label{prop:monad-frame-comparison}
The endofunctors \(\mathbb I\) and \(\mathbb T\) preserve the full subcategories \(\DLatFrm\) and \(\mathcal E(\DLatFrm)\), respectively. The functor \(\mathcal E\) intertwines their endofunctors, units, and multiplications by canonical natural isomorphisms. Equivalently, after identifying \(\DLatFrm\) with its image under \(\mathcal E\), the restriction of \(\mathbb T\) is the ideal-lattice monad \(\mathbb I\).

More explicitly, for every frame $F$:
\begin{enumerate}
\item the $k$-ideals of the semiring \(\mathcal E(F)\) are exactly the lattice ideals of $F$;
\item for lattice ideals $I,J\subseteq F$,
\[
I\odot_kJ=I\cap J;
\]
\item on objects and morphisms, \(\Idk\mathcal E\) is canonically \(\mathcal E\mathsf{Idl}\);
\item both units send \(x\) to \(\mathord\downarrow x\), and both multiplications send an ideal \(\mathcal J\) of lattice ideals to
\[
\bigcup_{I\in\mathcal J}I.
\]
\end{enumerate}
\end{proposition}

\begin{proof}
An ideal of the semiring \((F,\vee,\wedge,0,1)\) is closed under finite joins and under meets with arbitrary elements. Hence it is a lower set closed under finite joins, that is, a lattice ideal. Conversely, a lattice ideal is a semiring ideal. Since a subset of an additively idempotent semiring is a $k$-ideal exactly when it is downward closed, the two notions coincide.

For lattice ideals $I$ and $J$, every generator $x\wedge y$ of the ordinary semiring product lies in $I\cap J$, so
\(I\odot_kJ\subseteq I\cap J\). Conversely, if $z\in I\cap J$, then
\(z=z\wedge z\) is itself a product generator. Therefore
\(I\odot_kJ=I\cap J\).

If \(h\colon F\to G\) is a bounded lattice homomorphism, then $k$-ideal extension along \(\mathcal E(h)\) is the lattice ideal generated by $h[I]$. Thus the object and morphism parts agree with the ideal-lattice functor, naturally in $F$. The two units are the principal-ideal maps \(x\mapsto\mathord\downarrow x\).

Finally, let \(\mathcal J\) be a lattice ideal of \(\mathsf{Idl}(F)\). Its union is a lower set in $F$. If $x\in I$ and $y\in J$ for $I,J\in\mathcal J$, then
\(I\vee J\in\mathcal J\) and $x\vee y\in I\vee J$; hence the union is closed under finite joins. It is therefore a lattice ideal, and it is the join of the members of \(\mathcal J\). Consequently both monad multiplications are the union map. These identifications are natural and compatible with the monad laws.
\end{proof}

\begin{remark}\label{rem:monad-comparison-scope}
Proposition~\ref{prop:monad-frame-comparison} gives the exact overlap with the ideal-lattice monad studied in \cite{Johnstone,Razafindrakoto25}. It does not identify the two constructions globally. The monad \(\mathbb I\) is based on bounded distributive lattices and uses meet only after passing to frame objects; \(\mathbb T\) is based on complete idealic semirings and retains their given multiplication through the $k$-ideal product. General ideal-completion adjunctions from idempotent or ordered semirings to quantales appear in \cite{NishizawaFurusawa14,Fujii23}; the contribution here is the explicit subtractive monad, its frame comparison, and its Eilenberg--Moore description within the radical-ideal framework of this paper.
\end{remark}

\subsection{Eilenberg--Moore algebras and monadicity}\label{subsec:EM-algebras}

We now characterize the algebras of \(\mathbb T\) intrinsically and prove that the forgetful functor from integral commutative quantales is monadic. The missing arbitrary distributivity of a complete idealic semiring is exactly the structure encoded by a \(\mathbb T\)-algebra.

\begin{theorem}\label{thm:EM-quantales}
Let \(\mathbb T=\mathcal U\Idk\). The comparison functor
\[
\mathcal K\colon\QuantTop\longrightarrow\CIdRig^{\mathbb T},
\qquad
Q\longmapsto
\left(\mathcal UQ,
 I\longmapsto\bigvee_{x\in I}x\right),
\]
is an equivalence. Hence \(\mathcal U\) is monadic.

More explicitly, a complete idealic semiring \(\ddot S\) admits a \(\mathbb T\)-algebra structure if and only if its multiplication distributes over arbitrary joins, that is, if and only if \(\ddot S\) is the underlying complete idealic semiring of an object of \(\QuantTop\). When it exists, the algebra structure is unique and is given by
\[
\alpha_{\ddot S}(I)=\bigvee_{x\in I}x.
\]
A morphism of \(\mathbb T\)-algebras is precisely a quantale homomorphism.
\end{theorem}

\begin{proof}
Let \((\ddot S,\alpha)\) be a \(\mathbb T\)-algebra. Since a homomorphism of additively idempotent semirings is monotone, for every $k$-ideal $I$ and every \(x\in I\),
\[
x=\alpha(\mathord\downarrow x)\leq\alpha(I).
\]
Thus \(\bigvee I\leq\alpha(I)\). Conversely,
\(I\subseteq\mathord\downarrow(\bigvee I)\), so the unit law gives
\[
\alpha(I)\leq
\alpha(\mathord\downarrow(\bigvee I))
 =\bigvee I.
\]
Hence every algebra structure is forced to satisfy
\(\alpha(I)=\bigvee I\).

Let \(A\subseteq\ddot S\) and \(x\in\ddot S\). Using multiplicativity of $\alpha$ and Lemma~\ref{lem:k-generated-supremum}, we obtain
\begin{align*}
x\left(\bigvee_{a\in A}a\right)
 &=\alpha(\mathord\downarrow x)\,
   \alpha(K_{\ddot S}(A))\\
 &=\alpha\bigl((\mathord\downarrow x)\odot_kK_{\ddot S}(A)\bigr)\\
 &=\alpha\bigl(K_{\ddot S}(\{xa\mid a\in A\})\bigr)\\
 &=\bigvee_{a\in A}xa.
\end{align*}
Commutativity gives distributivity in both variables. Thus \(\ddot S\) is an integral commutative quantale.

Conversely, if $Q$ is such a quantale, the counit
\[
\alpha_Q(I)=\bigvee_{x\in I}x
\]
is a morphism in \(\CIdRig\). It is the underlying map of the counit at $Q$, so the comparison construction gives
\[
\alpha_Q\circ\eta_{\mathcal UQ}=1_{\mathcal UQ},
\qquad
\alpha_Q\circ\mathbb T(\alpha_Q)
 =\alpha_Q\circ\mu_{\mathcal UQ}.
\]
The first identity is the counit--unit triangular identity, and the second follows from naturality of the counit together with the definition
\(\mu=\mathcal U\varepsilon_{\Idk(-)}\). Hence \(\alpha_Q\) is a \(\mathbb T\)-algebra. The preceding paragraph shows that this is the only possible algebra structure on its underlying complete idealic semiring.

It remains to identify morphisms. Let
\(h\colon(\ddot S,\alpha)\to(\ddot T,\beta)\) be a \(\mathbb T\)-algebra morphism. For \(A\subseteq\ddot S\), preservation of $k$-ideal joins and the principal formula of Proposition~\ref{prop:Idk-functor} give
\begin{align*}
\mathbb T(h)(K_{\ddot S}(A))
 &=\mathbb T(h)\left(
   \mathop{\bigvee{}^k}_{a\in A}\mathord\downarrow a
   \right)\\
 &=\mathop{\bigvee{}^k}_{a\in A}\mathord\downarrow h(a)
 =K_{\ddot T}(h[A]).
\end{align*}
Therefore Lemma~\ref{lem:k-generated-supremum} and the algebra-morphism equation yield
\begin{align*}
h\left(\bigvee A\right)
 &=h\alpha(K_{\ddot S}(A))\\
 &=\beta\,\mathbb T(h)(K_{\ddot S}(A))\\
 &=\beta(K_{\ddot T}(h[A]))
 =\bigvee_{a\in A}h(a).
\end{align*}
Hence $h$ preserves arbitrary joins and is a quantale homomorphism. Conversely, an arbitrary-join-preserving semiring homomorphism satisfies
\[
h\left(\bigvee_{x\in I}x\right)
 =\bigvee_{x\in I}h(x)
 =\bigvee\mathbb T(h)(I),
\]
so it is a \(\mathbb T\)-algebra morphism. The comparison functor is therefore fully faithful and essentially surjective. This is the Eilenberg--Moore description of the adjunction \cite{EM65}.
\end{proof}

\begin{corollary}\label{cor:quantale-reflection-test}
For \(\ddot S\in\CIdRig\), the following are equivalent:
\begin{enumerate}
\item \(\ddot S\) is an integral commutative quantale;
\item the counit-type map
\[
\alpha_{\ddot S}\colon\Idk(\ddot S)\longrightarrow\ddot S,
\qquad
I\longmapsto\bigvee I,
\]
is a unital semiring homomorphism;
\item \(\ddot S\) carries a \(\mathbb T\)-algebra structure.
\end{enumerate}
In that case \(\alpha_{\ddot S}\) is the unique algebra structure.
\end{corollary}

\begin{proof}
Theorem~\ref{thm:EM-quantales} proves the equivalence of (1) and (3) and uniqueness. If (1) holds, the displayed map is the counit and hence a semiring homomorphism. If (2) holds, the calculation in the proof of Theorem~\ref{thm:EM-quantales} forces arbitrary distributivity, so (1) follows.
\end{proof}
\vskip 1cm

\noindent\textbf{Declarations.}

\vskip.25cm

\noindent\textbf{Funding.}
The first author was supported by the University Grants Commission (India) through a Senior Research Fellowship (ID: 211610013222/Joint CSIR--UGC NET June 2021). The first and third authors are also thankful to the DST-FIST Purse Programme (Programme no. SR/FST/MS-II/2021/101(C)) of the Department of Mathematics, Jadavpur University, Kolkata, India.

\smallskip
\noindent\textbf{Competing interests.}
The authors declare that they have no competing financial or non-financial interests that are directly or indirectly related to this work.

\smallskip
\noindent\textbf{Data availability.}
No datasets were generated or analyzed during the present study.

\end{document}